\documentclass{amsart}

\usepackage{latexsym}
\usepackage{graphicx, psfrag  ,  caption, subcaption}
\usepackage{amsfonts}
\usepackage{amsmath}
\usepackage{amsthm}
\usepackage{amssymb}
\usepackage{amscd}
\usepackage{enumerate}
\usepackage{rotating}
\usepackage{algorithm}
\usepackage{algorithmic}
\usepackage{mathptmx}
\usepackage{color} 
\usepackage{float}
\usepackage{multirow}
\usepackage{url}

\newtheorem{thm}{Theorem}[section]

\newtheorem{lem}[thm]{Lemma}
\newtheorem{prop}[thm]{Proposition}

\theoremstyle{definition}
\newtheorem{pdef}[thm]{Definition} 
\numberwithin{equation}{section}

\begin{document}

\title[Computing foliations]
{A unified approach to compute foliations, inertial manifolds, and tracking initial conditions}
\author{Y.-M. Chung}
\address{Department of Mathematics\\
Indiana University\\ Bloomington, IN 47405}

\author{M. S. Jolly$^\dagger$}
\address{$^\dagger$ corresponding author}
\email[Y-M. Chung]{yumchung@indiana.edu}
\email[M. S. Jolly]{msjolly@indiana.edu}

\date{\today}
\thanks{This work was supported in part by NSF grant number DMS-1109638.  The authors thank Ricardo Rosa for several stimulating discussions.}


\maketitle

\begin{abstract}
Several algorithms are presented for the accurate computation of the leaves in the foliation
of an ODE near a hyperbolic fixed point.  They are variations of a contraction mapping method  in \cite{rosa} to compute inertial manifolds, which represents a particular leaf in the unstable foliation.   Such a mapping is combined with one
for the leaf in the stable foliation to compute the tracking 
 initial condition for a given solution. The algorithms are demonstrated on the Kuramoto-Sivashinsky equation. 

\end{abstract}


\section*{Introduction}
\label{section introduction}

The Hartman-Grobman Theorem provides a local foliation for an ODE near a hyperbolic point; through each nearby (base) point there is a pair of leaves that define a conjugacy to the linearized flow.  In the classic case where the base point is the hyperbolic point itself, 
one leaf is its unstable manifold, the other its stable manifold.  In that case the leaves are invariant; for a general base point they are not.  
They can, however, be characterized by the exponential growth/decay rates of the differences between solutions that start on them.  If the gap in the spectrum of the linear part sufficiently dominates the Lipschitz constant for the nonlinear part in a large enough neighborhood, and the spectrum is positioned properly, the unstable manifold is an inertial manifold.  Each solution is attracted at an exponential rate to a particular "tracking" solution on the inertial manifold.  We present several algorithms for the accurate computation of the leaves in the foliation and as well as for the tracking 
 initial condition for a given solution. The algorithms are demonstrated on the Kuramoto-Sivashinsky equation, which is an amplitude model of thin film flow (see \cite{AIFKSE} and references therein).

There has been considerable analysis of foliations in the literature.
The finite dimensional case was studied in \cite{foliation1}, followed by
treatments for particular partial differential equations (PDEs)  in \cite{foliation2} and \cite{foliation3}.   The exponential tracking property of inertial manifolds was 
established in \cite{inertial2}.
We consider here the general Banach space setting, as in  \cite{foliation_Bates_Lu_Zeng_1}, \cite{foliation_Bates_Lu_Zeng_2}, and \cite{foliation_chow_lin_lu}, and follow the particular framework in \cite{foliation}.   

The computation of different elements in a foliation have been treated separately
with a variety of approaches.  The survey paper \cite{survey_dodel} discusses a great number 
of methods for classic stable and unstable 2D manifolds.  Those manifolds are global,
but are generally not the graphs of functions, unlike inertial manifolds,
which are usually assumed to be both \cite{inertial1}.   Approximate inertial manifolds
(see e.g. \cite{AIFKSE}) for dissipative PDEs are explicit expressions for the enslavement of the high modes in terms
of the low modes.  To reduce the error in their approximation one must increase the
number of low modes, and hence the dimension of the manifold.  Direct computation
of a global inertial manifold of fixed dimension was carried out in \cite{ComputingInertial2}, 
while accurate evaluation of the enslavement at individual low mode
inputs was achieved in \cite{JRT,ComputingInertial4}.  The first efforts to compute 
tracking initial conditions appear to be in \cite{Roberts1,Roberts2,Roberts3}, in the context of center manifolds.  Those methods involve an expansion of normal forms and iterative procedure to project onto a local basis of the tangent space of the center manifold.

We present here a unified approach to computing all these elements.
It is based on the Lyapunov-Perron contraction mapping on spaces
of functions in time used in \cite{foliation}, and outlined in Section 1.  The fixed point of the mapping is a particular solution
of the differential equation whose initial value provides an enslavement
of either the high modes in terms of the low, in the case of a leaf 
in the unstable foliation, or vise-versa in the the case of the stable
foliation.   These leaves are manifolds and graphs of functions whose
Lipschitz constants are less than one if the linear term in the equation
sufficiently dominates the nonlinear term.  Both functions are then combined
to form yet another contraction mapping whose fixed point is the intersection
of the manifolds.  In the particular case where the leaf in the unstable foliation
passes through a steady state (as in the inertial manifold), the fixed point of the combined contraction mapping is the tracking initial condition for any point on the leaf in the stable foliation.

The key then, is to discretize the Lyapunov-Perron contraction mapping.
This was first done in \cite{rosa} for the particular case of
the inertial manifold by using piecewise constant functions
over increasingly finer time intervals.  We adapt that approach to the stable foliation
which requires an inner integration of the differential equation, which 
happens to be forward in time, so it is practical for PDEs.   This is done in 
Section 2.  We then
consider improvements based Aitken's acceleration and Simpson's method.
The new methods are applied to the particular case of an inertial manifold 
in Section 3, and then combined for stable and unstable foliations
to compute tracking initial conditions in Section 4.   We wrap up in Section 5 
with a comparison of long time dynamics of 
computed tracking initial conditions and a linearly projected initial conditions.  Public domain software for this approach is available at \cite{FOLI8}.

\section{Assumptions and Foliation Theory}
We recall here the main features of foliation theory for ODEs, following the presentation by Castaneda and Rosa \cite{foliation}, where the proof can be found.  Let $X$ and $Y$ be Banach spaces and $Z=X\times Y$ be endowed with the norm $\|z\|=\|(x,y)\|=max\{\|x\|,\|y\|\}$.  Let $F: Z\longrightarrow X$ and $G: Z\longrightarrow Y$ and $A$, $B$ be two linear bounded operators defined on $X$ and $Y$.  Consider the following system of ODEs:
\begin{equation}
\begin{aligned}
\label{ODE}
&\dot{x} = Ax + F(x,y) \\
&\dot{y} = By + G(x,y).
\end{aligned}
\end{equation}
The assumptions are the following:
\begin{align}
& \|e^{At}\| \leq e^{\alpha t} \text{ and } \|e^{-Bt}\| \leq e^{-\beta t},\;\forall\;t\geq0, \text{ and for some } \alpha, \beta \in \mathbb{R}.\\
& H(z)=(F(z),G(z)) \text{ is Lipschitz with } Lip(H)\leq \delta \text{ and } H(0)=0.\\
& 2\delta < \beta - \alpha, \text{ called the spectral gap condition.}
\label{gap condition}
\end{align} 
Since $H(z)$ is Lipschitz and $A$, $B$ are bounded operators, it is known that the autonomous differential equation \begin{equation}\label{zODE} \dot{z} = Cz + H(z), \end{equation} where $C=A\times B$ possesses a global unique solution for any given initial condition.  We denote by $z(t,z_0)$, the solution of (\ref{zODE}) with initial condition $z(0)=z_0\in Z$, and $x(t,z_0)$ and $y(t,z_0)$, the $X$ and $Y$ components of $z(t,z_0)$, respectively.  Thus, $z(t,z_0)=(x(t,z_0),y(t,z_0))$, for all $t\in \mathbb{R}$ and $z_0\in Z$.

Typically, the nonlinear terms in most physical models are not globally Lipschitz.  If this is the case and the system is dissipative, the nonlinear terms can be truncated outside the absorbing ball.  More precisely, let $\rho$ be the radius of the absorbing ball.  Consider the prepared equation
\begin{equation}
	\dot{z} = Cz + H_{\rho}(z),
\end{equation}
where $H_{\rho} : Z \rightarrow Z$ which agrees with $H$ for $\|z\| \leq \rho$ and is globally Lipschitz.  Since the all the long time behavior of the original system is in the absorbing ball, such a preparation leaves that behavior unchanged.  One choice of $H_{\rho}$ is
\begin{equation}
\label{prepared equation}
H_{\rho} = \theta_{\rho}(\|z\|)H(z),
\end{equation} 
with
\begin{equation}
\theta_{\rho}(r) = \theta(\frac{r^2}{\rho^2}),\qquad\theta(s) = \begin{cases} 1,\;\text{for }s \in [0,1], \\ 2(s-1)^3 - 3(s-1)^2 + 1,\;\text{for }s\in[1,2], \\ 0,\;\text{for }s>2.\end{cases}
\end{equation}
%


The main result in \cite{foliation} is to characterize foliations by the exponential growth/decay of the difference of any two solutions with initial data in the same leaf.

\begin{thm}(Foliation Theorem)
\label{Foliation thm}
\begin{enumerate}[$\qquad(1)$]
\item {(Stable) $Z = \bigcup_{y \in Y} \mathcal{M}_y$, where
\begin{enumerate}[$(i)$]
 \item {$\mathcal{M}_y = \{z_0\in Z :\|z(t,z_0)-z(t,(0,y))\| \leq \|z_0 - (0,y)\|e^{(\alpha+\delta)t}$, $\forall\; t\geq0  \} $}
 \item {$\mathcal{M}_y = graph(\Phi_y)$, for some $\Phi_y: X \rightarrow Y$ such that
        $\Phi_y(0)=y.$}
\end{enumerate}}

\item {(Unstable)  $Z = \bigcup_{x \in X} \mathcal{N}_x$, where 
\begin{enumerate}[$(i)$]
 \item {$\mathcal{N}_x = \{z_0\in Z :\|z(t,z_0)-z(t,(x,0))\| \leq \|z_0 - (x,0)\|e^{(\beta-\delta)t}$, $\forall\; t\leq0  \} $}
 \item {$\mathcal{N}_x = graph(\Psi_x)$, for some $\Psi_x: Y \rightarrow X$ such that
        $\Psi_x(0)=x.$}
\end{enumerate}}

\item{Both $\Phi_{y}$ and $\Psi_{x}$ have Lipschitz constants bounded by $\delta/(\beta-\alpha-\delta)$.}

\end{enumerate}
\end{thm}

%
The terminology stable foliation (unstable foliation) comes from the classic case where $\alpha < 0 < \beta$, in which $\mathcal{M}_0$ ($\mathcal{N}_0$) are respectively the stable (unstable) manifolds of $0$.  The framework, however, also applies if $\alpha<\beta<0$ or $0<\alpha<\beta$.  Regradless, $\mathcal{M}_0$ and $\mathcal{N}_0$ are both invariant; $$z_0 \in \mathcal{M}_0 \Rightarrow z(t,z_0) \in \mathcal{M}_0\;\forall\;t\in\mathbb{R},$$ and similarly for $\mathcal{N}_0$.  As a consequence, given any initial data $z_0\in Z$, properties $(i)$ in Theorem \ref{Foliation thm} define distinguished solutions in $\mathcal{M}_0$ and $\mathcal{N}_0$.


\begin{prop} (Exponential Tracking)
\label{Exponential Tracking}
 Given $z_0 \in Z$, there exists a unique $z_0^{+} \in \mathcal{N}_0$
\begin{equation*}
 \|z(t,z_0)-z(t,z_0^{+})\| \leq e^{(\alpha+\delta)t} \|z_0-z_0^{+}\|,\;\;\forall\; t\geq 0.
\end{equation*}
\end{prop}

\begin{pdef}{\ }
 The solution $z(\cdot, z_0^+)$ is called the {\it{exponential tracking}} of $z(\cdot, z_0)$ and $z_0^+$ is called the {\it{tracking initial condition}} of $z_0$.
\end{pdef}
By Theorem \ref{Foliation thm}, the entire foliation is established. The intersection of a leaf from the stable foliation with one from the unstable foliation is a single element in $Z$; more precisely, for each $x_1\in X$, $y_1\in Y$, $\mathcal{M}_{y_1} \cap \mathcal{N}_{x_1} $ is a single element of $Z$.  This is proved in \cite{foliation} by showing that $$\Sigma:Z \rightarrow Z,\;\Sigma:(x,y)\mapsto (\Psi_{x_1}(y),\Phi_{y_1}(x)),$$ has a unique fixed point, which is the intersection of two manifolds.  In Section \ref{computing tracking initial condition}, we will implement an approximate of the $\Sigma$ map iteratively to compute the tracking initial condition. 
 
If \begin{equation}\label{condition on inertial}\alpha+\delta<0 \text{, and } \dim(Y)<\infty \end{equation} then $\mathcal{N}_0$ is an {\it{inertial manifold}}, i.e. an exponentially attracting, finite dimensional, Lipschitz manifold.  Some of the key features of the foliation are illustrated in Figure \ref{foliation for toy}(A).

\begin{figure}[h!]
 \psfrag{X}{\tiny$X$}
 \psfrag{Y}{\tiny$Y$}
 \psfrag{Mz0}{\tiny$\mathcal{M}_{z_0}$}
 \psfrag{Nz0}{\tiny$\mathcal{N}_{z_0}$}
 \psfrag{N0}{\tiny$\mathcal{N}_0$}
 \psfrag{M0}{\tiny$\mathcal{M}_0$}
 \psfrag{z0}{\tiny$z_0$}
 \psfrag{z+}{\tiny$z_0^+$}
 \psfrag{z1}{\tiny$z_1$}
 \psfrag{t0}{\tiny$t$}
 \psfrag{Mz0t0}{\tiny$\mathcal{M}_{z(t,z_0)}$}
 \psfrag{z(t0,z0)}{\tiny$z_0'$}
 \psfrag{z(t0,z1)}{\tiny$\;\;z_1'$}
 \psfrag{(Inertial Manifold)}{}

 \subcaptionbox{Foliations for \eqref{ode for toy problem}.\\$z_i':=z(t_0,z_i)$ for $i=0$, and $1$.}{
  \includegraphics[width=1.6in, height=1.8in]{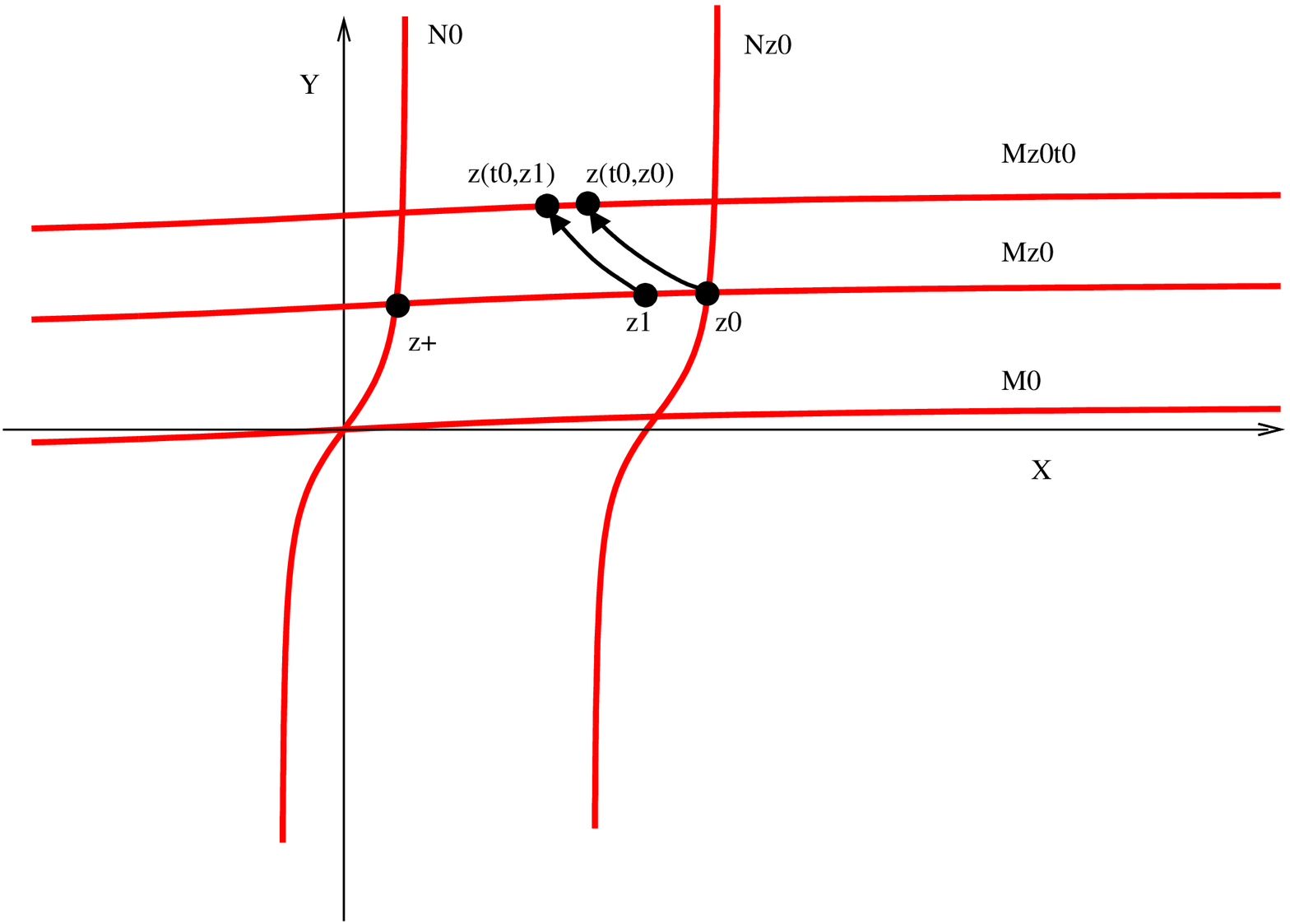}
 }
 \subcaptionbox{Number of iterations versus the absolute error.}{
 \includegraphics[width=1.6in, height=1.8in]{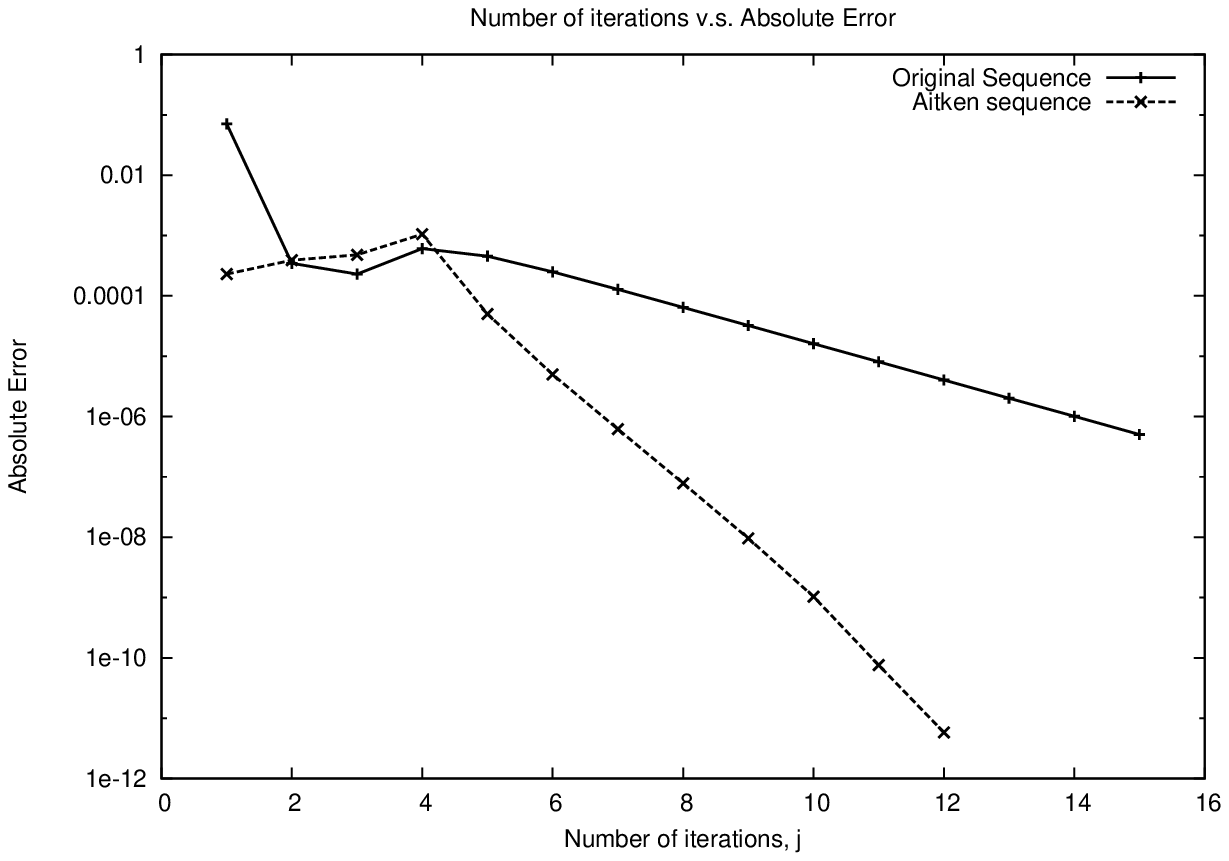}
 }
 \subcaptionbox{Number of iterations versus the absolute error.}{
 \resizebox{!}{!}{\begin{tabular} {|l|l|l|}
\hline
\small$j$ & \small SIMP & \small SIMPGS \\ \hline
\small 1	&	\small 0.2699E-1	&	\small 0.9353E-7	\\ \hline
\small 2	&	\small 0.3128E-5	&	\small 0.4832E-6	\\ \hline
\small 3	&	\small 0.2435E-5	&	\small 0.4831E-6	\\ \hline
\small 4	&	\small 0.4836E-6	&	\small 0.4831E-6	\\ \hline
\small 5	&	\small 0.4830E-6	&	\small 0.4831E-6	\\ \hline
\small 6	&	\small 0.4831E-6	&	\small 0.4831E-6	\\ \hline
\small 7	&	\small 0.4831E-6	&	\small 0.4831E-6	\\  \hline
\end{tabular}}
 } 
 
\caption{(A) $\mathcal{N}_0$ is the inertial manifold and $z_0^+$ is the tracking initial condition for $z_0$.  $\mathcal{M}_{z_0}$ is the leaf in the stable foliation through $z_0$ and has property that for any two points $z_0$, $z_1$ on $\mathcal{M}_{z_0}$ and any time $t>0$, $z(t,z_0)$ and $z(t,z_1)$ lie on the same manifold $\mathcal{M}_{z(t,z_0)}$.  (B) Performances of PWCONST and PWCONST along with the Aitken's accelerationtest for \eqref{ode for toy problem} with parameters $p=10$, $\tilde{z}_0=(1,1)$ and $\tilde{x}=3$. (C) Performances of SIMP and SIMPGS for the same settings as (B).}
\label{foliation for toy}
\end{figure}

\subsection{An Example: Test Problem}
\label{Toy Problem}
In this section, we will give an example that will be used to demonstrate the algorithms we developed in this article.  Consider the simplest system:
\begin{equation}
\label{linear system}
{d\over dt} {\tilde{x}\choose \tilde{y}}={-\tilde{x}\choose \tilde{y}} \quad \quad
\end{equation}
    The foliation of the linear system \eqref{linear system} consists of vertical and horizontal lines.  In order to obtain a nontrivial foliation, we will apply the transformation $T=T_2 \circ T_1$, where
\begin{equation}
 T_1{\tilde{x}\choose \tilde{y}}={\tilde{x} + \frac{\tilde{y}}{p\sqrt{1+\tilde{y}^2}}\choose \tilde{y}}, \quad T_2{\tilde{x}\choose \tilde{y}}={\tilde{x}\choose \tilde{y}+\frac{1}{p}\tan^{-1}(\tilde{x})}. \quad
\end{equation}
By an elementary calculation we obtain
\begin{equation}
 {x \choose y} = T{\tilde{x}\choose \tilde{y}}={\tilde{x} + \frac{\tilde{y}}{p\sqrt{1+\tilde{y}^2}}\choose \tilde{y} + \frac{1}{p}\tan^{-1}(\tilde{x} + \frac{\tilde{y}}{p\sqrt{1+\tilde{y}^2}})}. \quad
\end{equation}
 
After the transformation $T$, the new ODE can be written as:
\begin{equation}
\label{ode for toy problem}
\begin{aligned}
\frac{dx}{dt}=&-x+\frac{y-{\frac{1}{p}}tan^{-1}(x)}{p({1+({y-{\frac{1}{p}}tan^{-1}(x)})^2})^{1/2}}+\frac{y-{\frac{1}{p}}tan^{-1}(x)}{p({1+({y-{\frac{1}{p}}tan^{-1}(x)})^2})^{3/2}}\\
\frac{dy}{dt}=&y-\frac{1}{p}tan^{-1}(x)+\frac{1}{p(1+x^2)}[-x+\frac{y-{\frac{1}{p}}tan^{-1}(x)}{p({1+({y-{\frac{1}{p}}tan^{-1}(x)})^2})^{1/2}}+\\
&\frac{y-{\frac{1}{p}}tan^{-1}(x)}{p({1+({y-{\frac{1}{p}}tan^{-1}(x)})^2})^{3/2}}].
\end{aligned}
\end{equation}
Note that $p$ is a parameter that controls the Lipschitz constant.  In fact, we can characterize the complete foliation.  However, for the purpose of this article, we are interested in only the inertial manifold (invariant unstable manifold) and a leaf in the stable foliation.  

First consider the inertial manifold for the new system.  Since the invariant unstable manifold for the original system is the $y$-axis, its image under the map $T$ is the invariant unstable manifold for the new system.  That is
\begin{equation}
{x \choose y}=T{0 \choose \tilde{y}} ={ {\frac{\tilde{y}}{p\sqrt{1+\tilde{y}^2}}} \choose \tilde{y}+\frac{1}{p}\tan^{-1}(\frac{\tilde{y}}{p\sqrt{1+\tilde{y}^2}})}.
\end{equation}

Since the leaf in the stable foliation through $\tilde{z_0}$ for the original system is the horizontal line, $\tilde{y}=\tilde{y}_0$ for a given $\tilde{y}_0$, the transversal manifold for the new system is
\begin{equation}
{x \choose y}=T{\tilde{x} \choose \tilde{y}_0} = {\tilde{x} + \frac{1}{p\sqrt{1+\tilde{y}_0^2}} \choose 1 + \frac{1}{p}\tan^{-1}(\tilde{x} + \frac{1}{p\sqrt{1+\tilde{y}_0^2}})}. \quad
\end{equation}
By substitution, one obtains $$y = \tilde{y}_0 + \frac{1}{p} \tan^{-1}(x).$$

Given any initial condition $\tilde{z_0}=(\tilde{x_0},\tilde{y_0})$, the tracking initial condition for the original system is $(0,\tilde{y_0})$.  Thus, the tracking initial condition for the new system is
\begin{equation}\label{tracking for toy}T{0 \choose \tilde{y_0}} = {x \choose y} = {\frac{1}{p\sqrt{1+\tilde{y_0}^2}} \choose \tilde{y_0}+\frac{1}{p}\tan^{-1}(\frac{1}{p\sqrt{1+\tilde{y_0}^2}})}.\end{equation}

\section{Computation of the Stable Foliation}
\subsection{PWCONST algorithm}
\label{JRT section} 
In \cite{rosa}, an algorithm is developed for the accurate computation of inertial manifolds under the additional assumption \eqref{condition on inertial}.  The main idea is to find the fixed point of the contraction mapping
\begin{equation} \label{U map} \mathcal{U}(\psi,y)(t) = e^{tB}y + \int_{-\infty}^{t}e^{(t-s)A}F(\psi(s))ds - \int_{t}^{0}e^{(t-s)B}G(\psi(s))\;ds\;,\;\forall\; t\leq 0\;,\end{equation}
on the Banach space \begin{equation}\label{def Gsigma}\mathcal{G}_{\sigma}=\{\psi \in C((-\infty,0],Z); \; \|\psi\|_{\sigma}=\sup_{t\leq 0}e^{\sigma t}\|\psi(t)\|<\infty \},\;\sigma\in (\alpha+\delta,\beta-\delta).\end{equation} The inertial manifold is the graph of the function: \begin{equation} \Psi(y) = P\psi(0), \end{equation} where $\psi$ is the fixed point of $\mathcal{U}$ and $P$ is a projector from $Z$ onto $X$.  In the $j^{th}$ iteration, $\psi$ is approximated by $\psi^{(j)}$, a function that is piecewise constant on $N_j$ time intervals of length $h_j$.  It is shown in \cite{rosa} that $\psi^{(j)}\rightarrow \psi$ as $j \rightarrow \infty$ provided $N_jh_j \rightarrow \infty$ (e.g. $h_j=2^{-j}$ and $N_j = j2^j$).  In this approach the integrals in $\mathcal{U}$ can be evaluated explicitly.

To compute the leaf in the stable foliation through $z_0$, we follow the existence proof in \cite{foliation} and approximate the fixed point of the mapping
\begin{align}
\label{T map}
 \mathcal{T}_{z_0}(\varphi, x)(t) = e^{tA}x &+ \int_0^t e^{(t-s)A}[F(\varphi(s)+z(s,z_0))-F(z(s,z_0))]\;ds\\
                                            &- \int_t^{\infty} e^{(t-s)B}[G(\varphi(s)+z(s,z_0))-G(z(s,z_0))]\;ds,
\end{align} on \begin{equation}\label{def Fsigma} \mathcal{F}_{\sigma}=\{\varphi \in C([0,\infty],Z); \; \|\varphi\|_{\sigma}=\sup_{t\geq 0}e^{-\sigma t}\|\varphi(t)\|<\infty \}, \;\sigma\in (\alpha+\delta,\beta-\delta). \end{equation}   Let $\varphi$ be the fixed point of $\mathcal{T}$.  The leaf in the stable foliation through $z_0$ is the graph of the function: \begin{equation}\Phi_{z_0}(x) = y_0 + Q\varphi(x-x_0)(0), \end{equation}where $Q=I-P$.  We modify the algorithm in \cite{rosa} to fit $\mathcal{T}$.  The two main differences are: 1. one needs to solve for the ODE forward in time--- we use a 4-th order Runge-Kutta method (RK4); 2. an additional function evaluation is needed.  The rate of convergence is linear so that we can use the Aitken acceleration process, as discussed in the next section, to gain a better approximation.  Since piecewise constant functions are used in this algorithm, we denote it by PWCONST. 

\subsubsection{Aitken Acceleration}
Aitken's acceleration, also known as Aitken's $\Delta^2$ process, is used for accelerating the rate of convergence of a sequence.  The method works if one has a linear rate of convergence sequence. 
\begin{pdef}
Given a sequence $\{x_k\}_{k=0}^{\infty}$ in $\mathbb{R}$,  Aitken's acceleration sequence, $\{\mathcal{A}x_k\}_{k=0}^{\infty}$, is defined as
\begin{equation}
\mathcal{A}x_k = x_k - \frac{(\Delta x_k)^2}{\Delta^2x_k},
\end{equation}
where $\Delta x_k := x_{k+1}-x_{k}$ and $\Delta^2x_k:=\Delta x_{k+1}-\Delta x_k$.
\end{pdef}

This is the classic Aitken acceleration for a sequence in $\mathbb{R}$.  We will also apply a vector version found in \cite{Ait}.  
\begin{pdef}
Given a sequence  $\{z_k\}_{k=0}^{\infty}$ in $\mathbb{R}^n$, define \begin{equation*}\Delta z_k := (z_{k+1}-z_{k},...,z_{k+n}-z_{k+n-1})\;\text{and}\;\Delta^2 z_k:=\Delta z_{k+1}-\Delta z_k.\end{equation*}  Aitken's acceleration sequence for $\mathbb{R}^n$ is defined as follows:
\begin{equation}
\mathcal{A}z_k = z_k -(\Delta z_k) (\Delta^2 z_k)^{-1}(z_{k+1}-z_{k}).
\end{equation}
\end{pdef}
Note that both $\Delta z_k$ and $\Delta^2 z_k$ are matrices of size $n \times n$. 

The advantage of the Aitken sequence is that it converges much faster to the limit than the original sequence does.  Moreover, computing the Aitken sequence is much cheaper than computing the original sequence because it is applied to elements in phase space $z_k \in Z$, rather than (a discretized version of) the function space on which $\mathcal{U}$, $\mathcal{T}$ act.  It amounts to post-processing the original sequence.  A disadvantage is that it requires $n+k$ iterations of the original mapping $\mathcal{U}$ ($\mathcal{T}$) to produce $k$ terms in the Aitken sequence, which may be prohibitive if $n$, the dimension of $Z$, is large.  We will see these numerical results in the next section.   




\subsection{SIMP Algorithm}
\label{simp}
First, we give a recursive relation for $ \mathcal{T}_{z_0}(\varphi, x)$ in the time variable .  Second, we introduce the main algorithm.   At the end of this section, numerical results for the test problem \ref{ode for toy problem} will be given.

Let $\{t_i\}_0^{N}$ be the uniform partition of a time interval and let $h:=t_{i+1}-t_{i}$.  Let $P$ be the projector from $Z$ onto $X$, i.e. $Pz = x$, where $ z = (x,y)$ and $Q=I-P$.  

\subsubsection{A Recursive Relation}

We state the recursive relations in the following proposition.
\begin{prop}{\ }
\label{recursive relation}
   \begin{enumerate}[(i)]
     \item{For $i=1,\;2,\dots,\;N-1$,  the $X$-component of $\mathcal{T}_{z_0}(\varphi, x)$ is
                \begin{align*}
P \mathcal{T}_{z_0}(\varphi, x)(t_{i+1})=&e^{(t_{i+1}-t_{i-1})A}P \mathcal{T}_{z_0}(\varphi, x)(t_{i-1})+ \\
                                                                            & \int_{t_{i-1}}^{t_{i+1}}e^{(t_{i+1}-s)A}[F(\varphi(s)+z(s,z_0))-F(z(s,z_0))]ds,
                \end{align*}
               where $P \mathcal{T}_{z_0}(\varphi, x)(t_{0})=x$ and $$P \mathcal{T}_{z_0}(\varphi, x)(t_{1})=e^{t_1A}x +  \int_0^{t_1} e^{(t_1-s)A}[F(\varphi(s)+z(s,z_0))-F(z(s,z_0))]\;ds.$$
                }

  \item{For $i=N-1,\;N-2,\dots,\;1$, the $Y$-component of $\mathcal{T}_{z_0}(\varphi, x)$ is
             \begin{align*}
             Q\mathcal{T}_{z_0}(\varphi, x)(t_{i-1})=&e^{(t_{i-1}-t_{i+1})B} Q\mathcal{T}_{z_0}(\varphi, x)(t_{i+1})- \\
                                                       &\int_{t_{i-1}}^{t_{i+1}} e^{(t_{i-1}-s)B}[G(\varphi(s)+z(s,z_0))-G(z(s,z_0))]\;ds,
            \end{align*}
           where $Q\mathcal{T}_{z_0}(\varphi, x)(t_{N})=- \int_{t_N}^{\infty} e^{(t_N-s)B}[G(\varphi(s)+z(s,z_0))-G(z(s,z_0))]\;ds$ and $$Q\mathcal{T}_{z_0}(\varphi, x)(t_{N-1})=Q\mathcal{T}_{z_0}(\varphi, x)(t_{N})e^{(t_{N-1}-t_{N})B}-\int_{t_{N-1}}^{t_N} e^{(t_{N-1}-s)B}[G(\varphi(s)+z(s,z_0))-G(z(s,z_0))]\;ds.$$
}
   \end{enumerate}
\end{prop}
The proof of Proposition \ref{recursive relation} is similar to the derivation of the recursive algorithm in \cite{center_manifold}.

\subsubsection{The SIMP Algorithms}
The integrals appearing in the recursive relation can be better approximated by Simpson's rule.  In order to have a consistent order, one needs to adjust the first two approximations, i.e. $P\mathcal{T}_{z_0}(\varphi, x)(t_{0})$, $P\mathcal{T}_{z_0}(\varphi, x)(t_{1})$, $Q\mathcal{T}_{z_0}(\varphi, x)(t_{N})$, $Q\mathcal{T}_{z_0}(\varphi, x)(t_{N-1})$.   
 
Since $P\mathcal{T}_{z_0}(\varphi, x)(t_{0}) = x$, no error is introduced at $t=t_0$.  When $t=t_1$, \begin{equation*} P\mathcal{T}_{z_0}(\varphi, x)(t_{1})=e^{t_1A}x + \int_0^{t_1} e^{(t_1-s)A}[F(\varphi(s)+z(s,z_0))-F(z(s,z_0))]ds,\end{equation*}  which involves the integral we need to approximate so that the error is the same as the one introduced by Simpson's rule.  One can rewrite the integral as
\begin{equation*}
\int_{t_0}^{t_1} e^{(t_1-s)A}[F(\varphi(s)+z(s,z_0))-F(z(s,z_0))]\;ds
\end{equation*}
\begin{align}
\label{38 Simpson integral}
&=\int_{t_0}^{t_3} e^{(t_1-s)A}[F(\varphi(s)+z(s,z_0))-F(z(s,z_0))]\;ds\\
\label{Simpson integral}
&-\int_{t_1}^{t_3} e^{(t_1-s)A}[F(\varphi(s)+z(s,z_0))-F(z(s,z_0))]\;ds
\end{align}
For the integral (\ref{38 Simpson integral}), Simpson's $3/8$ rule is used. It has the same order error as the classic form of Simpson's rule, which is used for the integral (\ref{Simpson integral}).

The term $Q\mathcal{T}_{z_0}(\varphi, x)(t_{N})$ is the tail of the convergent improper integral.  We will give an estimate on $t_N$ so that the truncation error will be the same as the error of Simpson's rule.  To do so, we need an estimate on $\varphi$, the fixed point of $\mathcal{T}_{z_0}(\varphi, x)$.
\begin{lem}
Let $\varphi$ be the fixed point of $\mathcal{T}_{z_0}$.  Then $\|\varphi\|_{\sigma} \leq \frac{\|x\|}{1-\kappa}$, where $\kappa=max\{\frac{\delta}{\beta-\sigma}, \frac{\delta}{\sigma-\alpha}\}$.
\end{lem}
\begin{proof}
From \cite{foliation}, we have \begin{equation} \label{bound for phi} \|\mathcal{T}_{z_0}(\varphi)\|_{\sigma} \leq \|x\| + \kappa\|\varphi\|_{\sigma},  \end{equation}
for any $\varphi \in \mathcal{F}_{\sigma}$.  Since $\varphi$ is the fixed point of $\mathcal{T}_{z_0}$, the left hand side of \eqref{bound for phi} can be replaced by $\|\varphi\|_{\sigma}$.  By the gap condition \eqref{gap condition}, $\kappa<1$.  Thus, by direct calculation, one has $$\|\varphi\|_{\sigma} \leq \frac{\|x\|}{1-\kappa}.$$
\end{proof}

We are ready to give an estimate for $Q\mathcal{T}_{z_0}(\varphi, x)(t_{N})$.  Note that
\begin{align*}
\|Q\mathcal{T}_{z_0}(\varphi, x)(t_{N})\| &= \| \int_{t_N}^{\infty} e^{(t_N-s)B}[G(\varphi(s)+z(s,z_0))-G(z(s,z_0))]\;ds\| \\
                                          &\leq \int_{t_N}^{\infty}e^{(t_N-s)\beta}\delta|\varphi(s)|\;ds \leq \int_{t_N}^{\infty}e^{(t_N-s)\beta}\delta e^{\sigma s}\|\varphi\|_{\sigma}\;ds\\
                                          &=\delta\|\varphi\|_{\sigma}e^{t_N \beta}\int_{t_N}^{\infty}e^{(\sigma-\beta)s}\;ds
                                          =\frac{\delta}{\beta-\sigma}e^{\sigma t_N}\|\varphi\|_{\sigma}\\
                                                                             &\leq \frac{\delta}{\beta-\sigma}e^{\sigma t_N} \frac{\|x\|}{1-\kappa}
                                                                            \leq \frac{\kappa}{1-\kappa}e^{\sigma t_N}\|x\|.
\end{align*}
Let $h=t_1-t_0$.  Since the error for Simpson's rule is of order $\mathcal{O}(h^5)$, we take $$ \frac{\kappa}{1-\kappa}e^{\sigma t_N}\|x\| \leq h^5.$$ Hence, we choose $t_N$ so that \begin{equation} \label{condition on N}  e^{\sigma t_N} \leq \frac{1-\kappa}{\kappa\|x\|} h^5. \end{equation}
For $Q\mathcal{T}_{z_0}(\varphi, x)(t_{N-1})$, the technique is similar to the one used for $P\mathcal{T}_{z_0}(\varphi, x)(t_{1})$.  Rewrite $Q\mathcal{T}_{z_0}(\varphi, x)(t_{N-1})$ as two integrals.
\begin{equation*}
\int_{t_{N-1}}^{t_N} e^{(t_{N-1}-s)B}[G(\varphi(s)+z(s,z_0))-G(z(s,z_0))]ds
\end{equation*}
\begin{align}
\label{38 Simpson integral Q}
&=\int_{t_{N-3}}^{t_N} e^{(t_{N-1}-s)B}[G(\varphi(s)+z(s,z_0))-G(z(s,z_0))]ds\\
\label{Simpson integral Q}
&-\int_{t_{N-3}}^{t_{N-1}} e^{(t_{N-1}-s)B}[G(\varphi(s)+z(s,z_0))-G(z(s,z_0))]ds
\end{align}
As before, the integral (\ref{38 Simpson integral Q}) can be approximated by Simpson's $3/8$ rule; the integral (\ref{Simpson integral Q}) can be approximated by Simpson rule.  Hence, the error for $Q\mathcal{T}_{z_0}(\varphi, x)(t_{N-1})$ is of order $\mathcal{O}(h^5)$.

Before we present the algorithm, we remark on successive iteration.  We have to ensure that the initial guess, $\varphi^0$, is in the space $\mathcal{F}_{\sigma}$ defined in \ref{def Fsigma}.  In the PWCONST algorithm, $\psi^0$ is simply a constant function, which is in the space $\mathcal{G}_{\sigma}$ (defined in \ref{def Gsigma}) if we may take $\sigma<0$ (i.e. if $\alpha+\delta < 0 < \beta-\delta$).  However, in the case of the stable foliation, if $\varphi^0$ is constant, it may not be in $\mathcal{F}_{\sigma}$.  Since $\|\varphi^0\|_{\sigma}=\sup_{t\geq0} e^{-\sigma t} \|\varphi^0(t)\| $, $\|\varphi^0\|_{\sigma}$ is not bounded if $\sigma$ is negative.  Instead of constant $\varphi^0$, we take $$\varphi^0(x,t) = e^{\alpha t} (x-x_0).$$ Then 
\begin{equation*}
\|\varphi^0\|_{\sigma}=\sup_{t \geq 0} e^{-\sigma t} \|\varphi^0(t)\|=\sup_{t \geq 0} e^{(\alpha - \sigma)t} \|x-x_0\|<\infty
\end{equation*}
since $(\alpha - \sigma)$ is negative by the gap condition \eqref{gap condition}, where $\sigma \in (\alpha+\delta, \beta-\delta)$.

\begin{algorithm}
\caption*{SIMP algorithm}
\begin{algorithmic}[1]
\REQUIRE $z_0=(x_0,y_0)$, $x$, $h$, the step size, and $J$, the number of iterations.
\ENSURE $\Phi_{z_0}(x)$
\STATE{Choose $t_N$ by (\ref{condition on N}) and $t_i=i\times h$ for $i=0,\;,1,\;,\dots,\;N$.}
\STATE{$\varphi^0(x,t_i) = (x - x_0)e^{\alpha t_i}$.}
\STATE{Compute the solution of the ODE, $z(t_i,z_0)$ for $i=0,\;,1,\;,\dots,\;N$ by some ode solver.}
\STATE{Evaluate $F_i := F(z(t_i,z_0))$ and $G_i := G(z(t_i,z_0))$ for $i=0,\;,1,\;,\dots,\;N$.}
\FOR{$j=0 \to J-1$}
	\STATE{$F^j_i := F(\varphi^j(t_i)+z(t_i,z_0))-F_i$, for $i=0,\;,1,\;,\dots,\;N$}
	\STATE{$G^j_i := G(\varphi^j(t_i)+z(t_i,z_0))-G_i$, for $i=0,\;,1,\;,\dots,\;N$}
	\STATE{$P\mathcal{T}_{z_0}(\varphi^{j+1}, x)(t_{0})=x-x_0$}
	\STATE{$P\mathcal{T}_{z_0}(\varphi^{j+1}, x)(t_{1})=$\\$e^{t_1A}(x-x_0)+S38(e^{(t_1-t_k)A}F^j_k;k=0,1,2,3)-S(e^{(t_1-t_k)A}F^j_k;k=1,2,3)$}
	\STATE{$Q\mathcal{T}_{z_0}(\varphi^{j+1}, x)(t_{N})=0$}
	\STATE{$Q\mathcal{T}_{z_0}(\varphi^{j+1}, x)(t_{N-1})=$\\$-S38(e^{(t_{N-1}-t_k)A}G^j_k;k=N-3,N-2,N-1,N)+$\\$S(e^{(t_{N-1}-t_k)A}G^j_k;k=N-3,N-2,N-1)$}
	\FOR{$i=1 \to N-1$}
		\STATE{$P\mathcal{T}_{z_0}(\varphi^{j+1}, x)(t_{i+1})=$\\$e^{(t_{i+1}-t_{i-1})A}P\mathcal{T}_{z_0}(\varphi^j, x)(t_{i-1})+S(e^{(t_{i+1}-t_k)A}F^j_k;k=i-1,i,i+1)$}
		\STATE{$Q\mathcal{T}_{z_0}(\varphi^{j+1}, x)(t_{N-i-1})=$\\$e^{(t_{N-i-1}-t_{N-i+1})A}Q\mathcal{T}_{z_0}(\varphi^j, x)(t_{N-i+1})-$\\$S(e^{(t_{N-i-1}-t_k)A}G^j_k;k=N-i-1,N-i,N-i+1)$}
	\ENDFOR
\ENDFOR
\STATE{$\Phi_{z_0}(x) = y_0 + Q\mathcal{T}_{z_0}(\varphi^j, x)(0)$}
\end{algorithmic}
\end{algorithm}

For the notations in the SIMP algorithm,  \begin{align*}
S(f_k;k=i,i+1,i+2) &:= \frac{h}{3}(f_i + 4f_{i+1} + f_{i+2}), \\ S38(f_k;k=i,i+1,i+2,i+3) &:= \frac{3}{8}h(f_i + 3f_{i+1} + 3f_{i+2} + f_{i+3}),\end{align*} where $h=t_{i+1}-t_{i}$.  In particular for Simpson's rule in the $i$-loop, we have
\begin{align*}
&S(e^{(t_{i+1}-t_k)A}F^j_k;k=i-1,i,i+1)\\
&=\frac{h}{3}[e^{(t_{i+1}-t_{i-1})A}(F(\varphi^j(t_{i-1})+z(t_{i-1},z_0))-F(z(t_{i-1},z_0)))+\\
&\;4e^{(t_{i+1}-t_{i})A}(F(\varphi^j(t_{i})+z(t_{i},z_0))-F(z(t_{i},z_0)))+\\
&\;e^{(t_{i+1}-t_{i+1})A}(F(\varphi^j(t_{i+1})+z(t_{i+1},z_0))-F(z(t_{i+1},z_0)))].
\end{align*}
Since $P\mathcal{T}_{z_0}(\varphi^{j+1}, x)(t_{i-1})$ and $P\mathcal{T}_{z_0}(\varphi^{j+1}, x)(t_{i})$ are computed in the previous steps, we could, as in Gauss-Seidel iteration, use these to obtain a better approximation.  More precisely, we do the following:
\begin{align*}
&S(e^{(t_{i+1}-t_k)A}F^j_k;k=i-1,i,i+1)\\
&=\frac{h}{3}[e^{(t_{i+1}-t_{i-1})A}(F(P\varphi^{j+1}(t_{i-1})+x(t_{i-1},z_0),Q\varphi^{j}(t_{i-1})+y(t_{i-1},z_0))-F(z(t_{i-1},z_0)))+ \\
&\;4e^{(t_{i+1}-t_{i})A}(F(P\varphi^{j+1}(t_{i})+x(t_{i},z_0),Q\varphi^{j}(t_{i})+y(t_{i},z_0))-F(z(t_{i},z_0)))+\\
&\;e^{(t_{i+1}-t_{i+1})A}(F(\varphi^j(t_{i+1})+z(t_{i+1},z_0))-F(z(t_{i+1},z_0)))].
\end{align*}
A similar partial update can be done in reverse for $$S(e^{(t_{N-i-1}-t_k)A}G^j_k;k=N-i-1,N-i,N-i+1).$$  We call the resulting algorithm SIMPGS.

\subsection{Numerical results for the Test problem}

We apply PWCONST, PWCONST+Aitken, and SIMP to compute a leaf in the stable foliation for the test problem with $p=10$.  We take as the inputs for the algorithm \begin{equation*} (x_0,y_0)=T(1,1)=(1+\frac{1}{10\sqrt{2}},1+\frac{1}{10}tan^{-1}(1+\frac{1}{10\sqrt{2}}))\end{equation*} and $x=3+\frac{1}{10\sqrt{2}}$. The output should be \begin{equation*}y=1+\frac{tan^{-1}(x)}{10}=1+\frac{tan^{-1}(3+\frac{1}{10\sqrt{2}})}{10}. \end{equation*} 

In Figure \ref{foliation for toy}(B), the error of the original sequence decreases roughly by a factor of $1/2$ as $j$ increases; the error of the Aitken's sequence decreases roughly by a factor of $10^{-1}$.  In Figure \ref{foliation for toy}(C), observe that first, both algorithms converge and both errors are of the same order.  Comparing the two columns, there is a huge difference in the first two iterations.  In 1 or 2 iterations, the errors from SIMPGS seem to be saturated while errors from SIMP are saturated in 5 or 6 iterations.  This is typical. Different $h$ and different inputs give similar results.

\section{Improved Computation of Inertial Manifolds}
Both algorithms discussed in the previous section can be adapted to compute an inertial manifold.  In this section, we will show by the numerical evidence that these are the improved methods and we will apply those methods to the Kuramoto-Sivashinsky equation (KSE).  

The KSE with periodic and odd boundary conditions can be written
\begin{align}
\label{KSE}
 &\frac{\partial u}{\partial t} +  4\frac{\partial^4 u}{\partial\xi^4} + \gamma[\frac{\partial^2 u}{\partial \xi^2} + u\frac{\partial u}{\partial\xi}] = 0, \\
       & u(t,\xi) = u(t, \xi + 2\pi),\quad u(t,-\xi) = -u(t,\xi).
\end{align}
The solutions may be represented by the Fourier sine series
$$u(t,\xi) = \sum_{j=-\infty}^{\infty} u_j(t)e^{ij\xi} = \sum_{j=1}^{\infty}b_j(t)sin(j\xi),$$
where the reality and oddness conditions on $u$ give
$$u_{-j}=\bar {u_j} \text{ and } u_j = -\frac{ib_j}{2}.$$
For the tests in this section, we use a $16$ mode Galerkin approximation and take $dim(X)=dim(Y)=8$ and $\gamma=32$.

It has been shown that KSE (\ref{KSE}) has an inertial manifold (see \cite{inertialKSE}) and its lowest dimension has been studied in \cite{inertialKSE}, \cite{dim_KSE_inertial_Wang}, and \cite{dim_KSE_inertial}.  In particular at $\gamma=32$, it is shown in \cite{dim_KSE_inertial} that the {\it computed} global attractor is contained in a ball of radius $15$ (in the $L^2$-norm) which when used in the preparation in \eqref{prepared equation} yields an inertial manifold of dimension five.  In what follows we take $dim(X)=dim(Y)=8$, as the larger gap at this splittinbg leads to more rapid convergence for all methods.  Since a limit cycle is contained in the global attractor, which in turn is on the inertial manifold, we pick a test point,  $u_0=x_0+y_0$, on a limit cycle and pick  the low mode component, $y_0$, as an input of the algorithm and test how well we recover the high modes, $x_0$.

\subsection{PWCONST plus Aitken's Acceleration}

As we mentioned before, since the rate of convergence is linear in PWCONST, we can apply Aitken's $\Delta^2$ process to accelerate the convergence.
\begin{figure}[h!]
\subcaptionbox{Performances of PWCONST and PWCONST + Aitken's acceleration.}{
\begin{tabular} {|l|l|l|}
\hline
$j$	&	$x_0^j$	&	$\mathcal{A}x_0^j$   \\ \hline
1	&	0.616E-4	&	 N/A \\ \hline
2	&	0.321E-4	&	 N/A \\ \hline
3	&	0.146E-4	&	 N/A \\ \hline
4	&	0.707E-5	&	 N/A \\ \hline
5	&	0.348E-5	&	 N/A \\ \hline
6	&	0.173E-5	&	N/A  \\ \hline
7	&	0.862E-6	&	N/A  \\ \hline
8	&	0.430E-6	&	N/A  \\ \hline
9	&	0.214E-6	&	1.595E-11  \\ \hline
10	&	0.107E-6	&	1.629E-11  \\ \hline
11	&	0.537E-7	&	1.627E-11  \\ \hline
12	&	0.268E-7	&	1.627E-11  \\ \hline
13	&	0.134E-7	&	1.627E-11  \\ \hline
\end {tabular} }\qquad
\subcaptionbox{Performance of the SIMP algorithm.}{
\includegraphics[width=2.5 in]{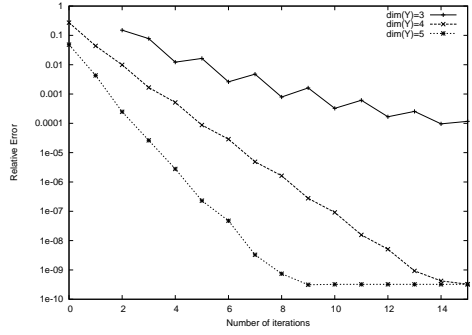}
}
\caption{Improved computation of the inertial manifold for the 16 mode Galerkin approximation of KSE with $\gamma=32$.  (A) Absolute error table for two algorithms.  $x_0^j$ ($\mathcal{A}x_0^j$) is generated by PWCONST (PWCONST+Aitken).    (B) We fix $(h,N)=(1e-6 ,100000)$ and vary $dim(Y)$.}
\label{improved inertial}
\end{figure}
The reason there are only 5 iterations in the Aitken's sequence  in Figure \ref{improved inertial}(A) is that in calculating one iteration in the Aitken's sequence, one needs $n+1$ iterations from the original sequence, where $n$ is the dimension of the elements of the sequence.  The error of Aitken's sequence is much better than the original sequence.  Another advantage of Aitken's $\Delta^2$ process is the cheap computation.  In order to obtain the similar error with the original sequence, one needs to compute more iterations, which is more expensive than computing the Aitken's sequence.  To observe the difference between these computational efforts, we compare the number of multiplications for the two algorithms, namely PWCONST and the Aitken's sequence ignoring multiplications needed for the evolution of nonlinear terms and the integrals.  For the PWCONST algorithm, the required number of multiplication is at least \begin{equation*} \Sigma_{j=1}^{J} j \times 2^j \times 6 \times dim(Z)\;. \end{equation*}  In this case, $dim(Z)=16$.  If we compute 10 terms in the PWCONST algorithm, the minimum number of multiplications is \begin{equation*} \Sigma_{j=1}^{10} j \times 2^j \times 6 \times dim(Z) = 1769664\;. \end{equation*}  On the other hand, to compute one term of the Aitken's sequence, one needs to compute $dim(X)$ terms in the PWCONST algorithm, solve a linear system, and calculate a matrix-vector multiplication.  Thus, the number of multiplications for computing two terms of the Aitken's sequence is \begin{equation}
 \sum_{j=1}^{dim(X)+1} j \times 2^j \times 6 \times dim(Z) + (dim(X)^3 + dim(X)^2) \times 2 = 787776. \end{equation}  

\subsection{The SIMP Algorithm}
To compute the inertial manifold and a leaf through $z_0$ in the stable foliation is to compute the fixed point of $\mathcal{U}$ in \eqref{U map} and $\mathcal{T}_{z_0}$ in \eqref{T map} on the Banach space $\mathcal{G}_{\sigma}$ in \eqref{def Gsigma} and $\mathcal{F}_{\sigma}$ in \eqref{def Fsigma} respectively.  These two maps are similar and in fact, if we formally replace $(z_0, \varphi, t, x, A, B, F, G)$ in the $\mathcal{T}$ map by $(0, \psi,-t,y,B,A,G,F)$ and drop the $z(s,z_0)$ terms,  we obtain the $\mathcal{U}$ map.  With this observation, we can easily modify the previous algorithms to compute the fixed point of the $\mathcal{U}$ map.  

There are two parameters in SIMP, namely step size, $h$, and the number of points, $N$.  Given $h$, one can choose $N$ by (\ref{condition on N}), though this requires an estimate on the Lipschitz constant of the nonlinear term, such can be found in \cite{dim_KSE_inertial}.  Here, we will choose $h$ and $N$ experimentally.  

Next, we investigate how the spectral gap affects the algorithm.  For the KSE, since the linear term is a diagonal matrix, the gap is the difference between two consecutive eigenvalues.

When $dim(Y)=1$ and $dim(Y)=2$, the sequences are convergent but they are not convergent to the inertial manifold.  When $dim(Y)=3$, the convergence is slow; it took about $50$ iterations to reach the saturated error, which is about $10^{-8}$.  The reason could be that the gap condition is barely satisfied for the region of phase space visited by the algorithm.  As we increase $\dim(Y)$ and hence the gap, the convergence is faster as Figure \ref{improved inertial}(B) shows.

\section{Computation of Tracking Initial Conditions}
\label{computing tracking initial condition}
\subsection{Algorithm for tracking initial condition}
\label{section algorithm}
It is shown in \cite{foliation} that the exact tracking initial condition of a base point $(x_1,y_1)$ is the fixed point for the mapping
\begin{equation}
\Sigma:\;(x,y) \mapsto(\Psi_{x_1}(y), \Phi_{y_1}(x)) \;.
\end{equation}
    We will fix $j_1$ and $j_2$ and iterate the map \begin{equation} \Sigma^{j_1,j_2}:\;(x,y)\mapsto (\Psi^{j_1}{_{x_1}}(y), \Phi^{j_2}_{y_1}(x)). \end{equation}  As is the case for $\Sigma$,  $\Sigma^{j_1,j_2}$ is a contraction mapping and the fixed point, $(x^*, y^*)$, is the intersection of the two manifolds that are the graphs of $\Psi^{j_1}{_{x_1}}$ and $\Phi^{j_2}_{y_1}(x)$.  

\subsection{Convergence of the algorithm}
The following results show that the algorithm for the tracking initial condition converges under a stronger gap condition.
\begin{lem}
\label{lemma phi}
$\|\varphi^j(x_1) - \varphi^j(x_2)\|_{\sigma} \leq \frac{1}{1-\kappa} \|x_1 - x_2\|$, where $\kappa=max\{\frac{\delta}{\beta-\sigma}, \frac{\delta}{\sigma-\alpha}\}$ and \\ $\|\varphi\|_{\sigma} := \sup_{t\geq0}e^{-\sigma t}\|\varphi(t)\|$.
\end{lem}

\begin{proof}

For $j \geq 1$, define $\varphi^{j}$ recursively by \begin{equation*}\varphi^{j}(x)(t) = \mathcal{T}_{z_0}(\varphi^{j-1}, x)(t),\;\text{where}\; \varphi^0(x)(t)=e^{\alpha t}(x-x_0),\;\forall t\in [0,\infty).\end{equation*}

Now, let $t\in [0,\infty)$ and consider
\begin{align*}
&\|\varphi^j(x_1)(t) - \varphi^j(x_2)(t)\|= \|\mathcal{T}_{z_0}(\varphi^{j-1}, x_1)(t)-\mathcal{T}_{z_0}(\varphi^{j-1}, x_2)(t)\| \\
= &\|e^{tA}(x_1-x_2) +  \int_0^t e^{(t-s)A}[F(\varphi^{j-1}(x_1)(s)+z(s,z_0))-F(\varphi^{j-1}(x_2)(s)+z(s,z_0))]ds \\
  &-\int_t^{\infty} e^{(t-s)B}[G(\varphi^{j-1}(x_1)(s)+z(s,z_0))-G(\varphi^{j-1}(x_2)(s)+z(s,z_0))]ds\|\\
\leq & max\{e^{\alpha t} \|x_1-x_2\| + \int_0^t e^{(t-s)\alpha}\delta\|\varphi^{j-1}(x_1)(s)-\varphi^{j-1}(x_2)(s)\|ds,\\
     & \int_t^{\infty} e^{(t-s)\beta}\delta\|\varphi^{j-1}(x_1)(s)-\varphi^{j-1}(x_2)(s)\|ds\}\\
\end{align*}
Let $\sigma\in(\alpha+\delta, \beta-\delta)$ and multiply the last inequality above by $e^{-\sigma t}$ to obtain:\\
$e^{-\sigma t}\|\varphi^j(x_1)(t) - \varphi^j(x_2)(t)\|$
\begin{align*}
\leq max\{&e^{(\alpha-\sigma) t}\|x_1-x_2\|+e^{-\sigma t}\int_0^t e^{(t-s)\alpha}\delta\|\varphi^{j-1}(x_1)(s)-\varphi^{j-1}(x_2)(s)\|ds,\\
          &e^{-\sigma t}\int_t^{\infty} e^{(t-s)\beta}\delta\|\varphi^{j-1}(x_1)(s)-\varphi^{j-1}(x_2)(s)\|ds\}\\
\leq max\{&e^{(\alpha-\sigma)t}\|x_1-x_2\|+\delta\|\varphi^{j-1}(x_1) - \varphi^{j-1}(x_2)\|_{\sigma}\int_0^t e^{(\alpha-\sigma)(t-s)}ds,\\
          &\delta\|\varphi^{j-1}(x_1) - \varphi^{j-1}(x_2)\|_{\sigma}\int_t^{\infty} e^{(\beta-\sigma)(t-s)}ds\}\\
= max\{& e^{(\alpha-\sigma)t}\|x_1-x_2\| + \frac{\delta}{\sigma-\alpha}\|\varphi^{j-1}(x_1) - \varphi^{j-1}(x_2)\|_{\sigma},\frac{\delta}{\beta-\sigma}\|\varphi^{j-1}(x_1) - \varphi^{j-1}(x_2)\|_{\sigma}\}\\
\leq \|x_1-&x_2\| + \kappa\|\varphi^{j-1}(x_1) - \varphi^{j-1}(x_2)\|_{\sigma} 
\end{align*}
Now, since the right hand side of the inequality does not depend on $t$, take the supremum over $t$ and obtain
\begin{align*}
\|\varphi^{j}(x_1) - \varphi^{j}(x_2)\|_{\sigma} &\leq \|x_1-x_2\| + \kappa\|\varphi^{j-1}(x_1) - \varphi^{j-1}(x_2)\|_{\sigma} \\
                                                 &\leq \|x_1-x_2\| + \kappa\|x_1-x_2\|+\kappa^2\|\varphi^{j-2}(x_1) - \varphi^{j-2}(x_2)\|_{\sigma}\\
                                                 &\leq \|x_1-x_2\| + \kappa\|x_1-x_2\| + \ldots + \kappa^j \|\varphi^{0}(x_1) - \varphi^{0}(x_2)\|_{\sigma} \\
                                                 &=(1+\kappa+\kappa^2+\ldots+\kappa^j)\|x_1-x_2\|
                                                 \leq \frac{1}{1-\kappa} \|x_1-x_2\|. 
\end{align*}
\end{proof}

Since $\Phi_{z_0}(x) = y_0 + Q\varphi(0)$, we can obtain an estimate for $\Phi^{j}$.  For the exact manifold, from \cite{foliation} we have $Lip(\Phi_{z_0}) \leq \frac{\delta}{\beta-\alpha-\delta}$.
\begin{lem}
\label{lemma for Phi}
$$\|\Phi^{j}(x_1)-\Phi^{j}(x_2)\|\leq \frac{\delta}{\beta-\sigma}\frac{1}{1-\kappa}\|x_1-x_2\|.$$
\end{lem}
\begin{proof}
Apply Lemma \ref{lemma phi} to find that
\begin{align*}
\|\Phi^{j}(x_1) - \Phi^{j}(x_2)\| &=\|\int_{0}^{\infty} e^{-sB}[G(\varphi^{j-1}(x_1)(s)+z(s,z_0))-G(\varphi^{j-1}(x_2)(s)+z(s,z_0))]ds \| \\
                                  &\leq \int_{0}^{\infty} e^{-s\beta} \delta \|\varphi^{j-1}(x_1)(s) - \varphi^{j-1}(x_2)(s)\|ds \\
                                  &=\int_{0}^{\infty} e^{-s\beta} e^{s\sigma}\delta e^{-s\sigma}\|\varphi^{j-1}(x_1)(s) - \varphi^{j-1}(x_2)(s)\|ds \\
                                  &\leq \delta\|\varphi^{j-1}(x_1)-\varphi^{j-1}(x_2)\|_{\sigma} \int_{0}^{\infty}e^{(\sigma-\beta)s}ds \\
                                  &= \frac{\delta}{\beta-\sigma} \|\varphi^{j-1}(x_1)-\varphi^{j-1}(x_2)\|_{\sigma} 
                                  \leq \frac{\delta}{\beta-\sigma}\frac{1}{1-\kappa}\|x_1-x_2\|.
\end{align*}
\end{proof} 
 
We have an analogous result for $\Psi$, the proof of which is similar.
\begin{lem}
\label{lemma for Psi}
$$\|\Psi^{j}(y_1)-\Psi^{j}(y_2)\|\leq \frac{\delta}{\sigma-\alpha}\frac{1}{1-\kappa}\|y_1-y_2\|.$$
\end{lem}

Since we now have both estimates for $\Phi^{j}$ and $\Psi^{j}$, ready to show  that $\Sigma^{j_1,j_2}$ is a contraction mapping.
\begin{prop}
\label{Lip for Sigma}
If \begin{equation}\label{stronger gap condition} 4\delta < \beta - \alpha, \end{equation} $\Sigma^{j_1,j_2}$ defined above is a contraction mapping with \begin{equation*}Lip(\Sigma^{j_1,j_2})\leq \frac{\kappa}{1-\kappa}. \end{equation*}
\end{prop}

\begin{proof}
By Lemma \ref{lemma for Phi} and \ref{lemma for Psi},
\begin{align*}
\|\Sigma^{j_1,j_2}(x_1,y_1)-\Sigma^{j_1,j_2}(x_2,y_2)\| &=\;max\{\|\Phi^{j}(x_1)-\Phi^{j}(x_2)\|, \|\Psi^{j}(y_1)-\Psi^{j}(y_2)\|\} \\
                                                        &\leq \frac{\kappa}{1-\kappa}\;max\{\|x_1-x_2\|,\|y_1-y_2\|\}. \\
\end{align*}
The condition $\frac{\kappa}{1-\kappa}<1$ is equivalent to $\kappa=\;max\{\frac{\delta}{\beta-\sigma},\frac{\delta}{\sigma-\alpha} \}<1/2$, i.e., $2\delta<\sigma-\alpha$ and $2\delta<\beta-\sigma$.  The last two inequalities are equivalent to \eqref{stronger gap condition} for $\sigma = \frac{\alpha+\beta}{2}$.  $\Box$
\end{proof}

\begin{prop}
$\Sigma^{j_1,j_2}$ has as its fixed point the intersection of the graphs of $\Psi^{j_1}$ and $\Phi^{j_2}$.
\end{prop}
The proof is as for $\Sigma$ in Castaneda Rosa \cite{foliation}.

\subsection{Error estimate}
Let $z_0^*$ be the fixed point of $\Sigma^{j_1,j_2}$, and $z_0^+$ the exact tracking initial condition (fixed point of $\Sigma$).  We seek an estimate for $\|z_{0}^{+} - z_0^{*}\|$.  We start with two lemmas.

\begin{lem}
Let $z_0\in Z$.  For any $x\in X$ and positive integer $j$, we have
\begin{equation*}\|\varphi(x)-\varphi^{j}(x)\|_{\sigma} \leq \frac{\kappa^j}{1-\kappa}\|\varphi^1(x)-\varphi^0(x)\|_{\sigma},\end{equation*} where $\varphi$ is the fixed point of $\mathcal{T}_{z_0}(\varphi, x)$ and $\kappa=max\{\frac{\delta}{\beta-\sigma}, \frac{\delta}{\sigma-\alpha}\}$.
\end{lem}
\begin{proof}
  Let $m$ be an integer such that $m>j$.  From \cite{foliation}, we know that $Lip(\mathcal{T}) \leq \kappa$ and by the gap condition, $\kappa<1$.
\begin{align*}
\|\varphi^{m}(x)-\varphi^{j}(x)\|_{\sigma} & \leq \Sigma_{k=j+1}^{m}\|\varphi^{k}(x)-\varphi^{k-1}(x)\|_{\sigma}  \\
&\leq \Sigma_{k=j+1}^{m} \kappa^{k-1} \|\varphi^{1}(x)-\varphi^{0}(x)\|_{\sigma} 
\leq \frac{\kappa^{j}}{1-\kappa}\|\varphi^{1}(x)-\varphi^{0}(x)\|_{\sigma}.
\end{align*}
 Since the last inequality is independent of $m$, take $m\rightarrow \infty$ to obtain
\begin{equation*} \|\varphi(x)-\varphi^{j}(x)\|_{\sigma} \leq \frac{\kappa^j}{1-\kappa}\|\varphi^1(x)-\varphi^0(x)\|_{\sigma}.\end{equation*}
\end{proof}

Since $\Phi_{z_0}(x) := y_0 + Q\varphi(x)(0)$, the estimate for $\|\Phi(x)-\Phi^{j}(x)\| $ is straightforward.
\begin{lem} 
\label{Phi-Phij}
Under the same assumption above, one has
$$\|\Phi(x)-\Phi^{j}(x)\| \leq \frac{\kappa^j}{1-\kappa}\|\varphi^1(x)-\varphi^0(x)\|_{\sigma}.$$
\end{lem}
\begin{proof}
By the definition of $\Phi$ and the previous lemma, 
\begin{align*}
\|\Phi(x)-\Phi^{j}(x)\| &\leq \|Q\varphi(x)(0)-Q\varphi^{j}(x)(0)\|\\
                                     &\leq\|\varphi(x)-\varphi^{j}(x)\|_{\sigma}\leq\frac{\kappa^j}{1-\kappa}\|\varphi^1(x)-\varphi^0(x)\|_{\sigma}.
\end{align*}
\end{proof}

We have an analogous result for $\Psi$, the proof of which is similar.
\begin{lem}
\label{Psi-Psij}
Under the same assumption above, one has
$$\|\Psi(y)-\Psi^{j}(y)\| \leq \frac{\kappa^j}{1-\kappa}\|\psi^1(y)-\psi^0(y)\|_{\sigma}.$$
\end{lem}

By Lemma \ref{Phi-Phij} and \ref{Psi-Psij}, one can deduce that
\begin{equation}
\label{estimate for sigma}
\|\Sigma(z) - \Sigma^{j_1,j_2}(z)\| \leq \frac{\kappa^j}{1-\kappa} \max\{\kappa^{j_1-j}\|\varphi^1(x)-\varphi^0(x)\|_{\sigma},\; \kappa^{j_2-j}\|\psi^1(y)-\psi^0(y)\|_{\sigma}\},
\end{equation}
where $j=\min\{j_1,j_2\}$. We are ready to give an estimate for $\|z_{0}^{+} - z_0^{*}\|$.
\begin{prop}
\label{error_estimate}
Let $z_0^+$ be the fixed point of $\Sigma$ and $z_0^*$ be the fixed point of $\Sigma^{j}$.  Then
\begin{equation*}\|z_{0}^{+} - z_0^{*}\| \leq c_{j_1,j_2}\frac{\kappa^j}{1-2\kappa}, \end{equation*}
where $j=\min\{j_1,j_2\}$ and $c_{j_1,j_2} =  max\{\kappa^{j_1-j}\|\varphi^1(x)-\varphi^0(x)\|_{\sigma},\; \kappa^{j_2-j}\|\psi^1(y)-\psi^0(y)\|_{\sigma}\}$.
\end{prop}
\begin{proof}
By (\ref{estimate for sigma}) and Proposition \ref{Lip for Sigma}, 
\begin{align*}
\|z_{0}^{+} - z_0^{*}\| &=\|\Sigma(z_{0}^{+})-\Sigma^{j_1,j_2}(z_0^{*})\|
                        \leq \|\Sigma(z_{0}^{+})-\Sigma^{j_1,j_2}(z_{0}^{+})\|+ \|\Sigma^{j_1,j_2}(z_{0}^{+})-\Sigma^{j_1,j_2}(z_0^{*})\|\\
                        &\leq c_{j_1,j_2} \frac{\kappa^{j}}{1-\kappa}+ \frac{\kappa}{1-\kappa}\|z_{0}^{+} - z_0^{*}\|.
\end{align*}
Therefore, we have \begin{equation*} \|z_{0}^{+} - z_0^{*}\| \leq
c_{j_1,j_2}\frac{\kappa^j}{1-2\kappa}. \end{equation*}
\end{proof}

For simplicity, we denote $\Sigma^j := \Sigma^{j,j}$ in the following numerical tests.

\subsection{Application}
\subsubsection{Test Problem}
In this section, we demonstrate the computation of the tracking initial condition for the test problem \eqref{ode for toy problem}.  In this test, $(x_0,y_0)=(1,1)$ and $p=10$ and hence by (\ref{tracking for toy}), the exact tracking initial condition is \begin{equation}{u_0^+ \choose v_0^+} = {\frac{1}{10\sqrt{2}} \choose 1+\frac{1}{10}\tan^{-1}(\frac{1}{10\sqrt{2}})} \approx {7.0711e-02 \choose 1.0070}.\end{equation}

By Proposition \ref{Lip for Sigma}, for fixed $j$, the Lipschitz constant for $\Sigma^j$ is \begin{equation*} \frac{\kappa}{1-\kappa}. \end{equation*}  In the test problem, $\alpha=-1$, $\beta=1$, and $\delta<1/p$ and hence $\kappa < 1/p$.  In this case, $p=10$, and therefore $Lip(\Sigma^j) < 1/9$.  As Table \ref{tracking table j6} (A) shows, in the first two iterations we observe that  the error decreases roughly by a factor of $0.01$ which confirms $Lip(\Sigma^j) < 1/9$.

To verify Proposition \ref{error_estimate}, we will vary $j$ and observe how the error depends on $j$.  As in Table \ref{tracking table j6} (B), we can see that the error decreases roughly by a factor of $0.1$ as $j$ increases which also confirms Proposition \ref{error_estimate} that the rate of the convergence is approximately \\$\kappa (\approx 1/10)$.
\begin{figure}
  \subcaptionbox{}{\begin{tabular}{|l|l|}
    \hline
$z^{i+1}$ & Error \\
    \hline
$z^0$  &  3.5429         \\  \hline
$z^1$  &  2.4E-2 \\  \hline
$z^2$  &  8.4E-5 \\  \hline
$z^3$  &  2.9E-7\\  \hline
$z^4$  &  1.0E-9\\  \hline
$z^5$  &  3.3E-11\\  \hline
  \end{tabular}}
  \subcaptionbox{}{\begin{tabular}{|l|l|}
\hline
$\Sigma^j$ & Error \\ \hline
$\Sigma^1$	&	4.97E-4		\\ \hline
$\Sigma^2$	&	8.49E-6		\\ \hline
$\Sigma^3$	&	1.85E-7		\\ \hline
$\Sigma^4$	&   4.50E-9		\\ \hline
$\Sigma^5$	&	1.16E-10		\\ \hline
$\Sigma^6$	&	3.25E-11		\\ \hline
      \end{tabular}}
\subcaptionbox{}{ \psfrag{X}{\tiny$X$}
 \psfrag{Y}{\tiny$Y$}
 \psfrag{z0}{\tiny$z_0$}
 \psfrag{zinitial}{\tiny$z^0$}
 \psfrag{z1}{\tiny$z^1$}
 \psfrag{z2}{\tiny$z^2$}
 \psfrag{z3}{\tiny$\;z^3$}
 \psfrag{phi}{\tiny$\text{graph}(\Phi_{z_0})$}
 \psfrag{psi}{\tiny$\text{graph}(\Psi_0)$}
 \includegraphics[width=2.5 in]{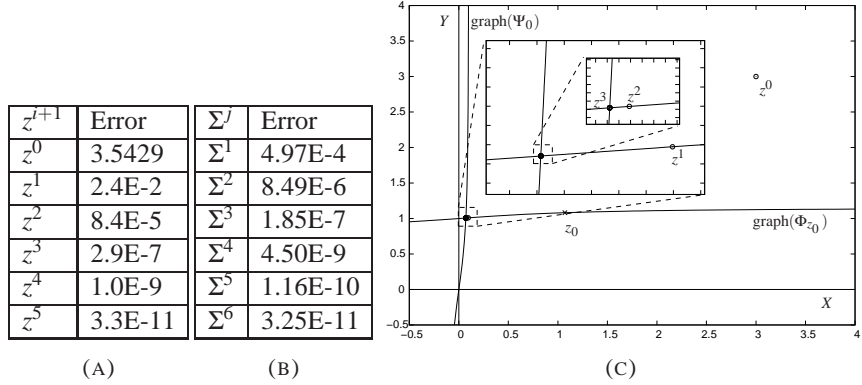}}
 
\caption{(A) This table shows the convergence of the tracking initial condition for \eqref{ode for toy problem} using SIMP for fixed $j=15$ and $z^{i+1}:=\Sigma^{j}(z^i)$ with $h=0.01$.  (B) This table shows the relation between $j$ and the absolute error. (C) Convergence for data in (A) in phase space. }
\label{tracking table j6}
\end{figure}

\subsubsection{Approximate Inertial Form of the KSE}
\label{section tracking AIFKSE}
We express the KSE in the functional form:
\begin{equation*}
 \frac{du}{dt} + Lu + R(u) = 0, \;\; u\in \mathcal{H},
\end{equation*}
where $\mathcal{H}$ is an appropriate Hilbert space (see \cite{AIFKSE}). The linear operator $L$ is given by \begin{equation*}Lu=4\frac{\partial^4 u}{\partial\xi^4} + \gamma\frac{\partial^2 u}{\partial \xi^2}\end{equation*} along with periodic, odd boundary conditions.  The remaining terms are then collected in $R$.  The infinite-dimensional phase space $\mathcal{H}$ is split into low- and high-wavenumber modes by means of the projectors 
\begin{equation*}
 P=P_n:\mathcal{H}\rightarrow span\{sin(\xi), sin(2\xi), \dots, sin(n\xi) \},\; Q = Q_n = I - P_n.
\end{equation*}
Thus, $u = p + q$, where $p = Pu$ and $q = Qu$.  We will use the approximate inertial manifold (see \cite{FMT}), $\Phi_1(p) = -L^{-1}QR(p)$ .  The approximate inertial form is
\begin{equation}
\label{AIFKSE equ}
 \frac{dp}{dt} + Lp + PR(p + \Phi_1(p)) = 0,
\end{equation}
which in our foliation framework would mean $z=p$, $C=-L$, and $H(z) = -PR(z+\Phi_1(z))$.  Compare to the Galerkin approximation which amounts to replacing $\Phi_1$ with $\Phi_0\equiv 0$.
We fix $n=3$ as it was demonstrated in \cite{AIFKSE} that this is sufficient for \eqref{AIFKSE equ} to capture the long time dynamics of the KSE for $\gamma\in[0,36]$. There are two reasons convergence should be slower for this reduced system.  First, the gap in the eigenvalues is smaller.  Second, the composition of the nonlinear term with itself in \eqref{AIFKSE equ} makes for a larger Lipschitz constant. 
%
%

Given two initial conditions for \eqref{AIFKSE equ}, denoted by $z_1$ and $z_2$, which are close to each other but on opposite sides of a separatrix, we will construct the leaves of the stable foliation through each point.  Since each leaf is an equivalence class, the points on the same leaf should have the same long time behavior.  This property is demonstrated in Figure \ref{fp_and_lc}. 

\begin{figure}[h!]
\psfrag{p1}{\tiny$p_1$}
\psfrag{p2}{\tiny$p_2$}
\psfrag{p3}{\tiny$p_3$}
\psfrag{ztz1}{\tiny$z(\cdot,z_1)$}
\psfrag{zta1}{\tiny$z(\cdot,a_1)$}
\psfrag{ztb1}{\tiny$z(\cdot,b_1)$}
\psfrag{ztc1}{\tiny$z(\cdot,c_1)$}
\psfrag{ztz2}{\tiny$z(\cdot,z_2)$}
\psfrag{zta2}{\tiny$z(\cdot,a_2)$}
\psfrag{ztb2}{\tiny$z(\cdot,b_2)$}
\psfrag{ztc2}{\tiny$z(\cdot,c_2)$}

\psfrag{Mz1}{\tiny$\mathcal{M}_{z_1}$}
\psfrag{Mz2}{\tiny$\mathcal{M}_{z_2}$}

\begin{center}
\subfloat[ ]{\includegraphics[scale=.5]{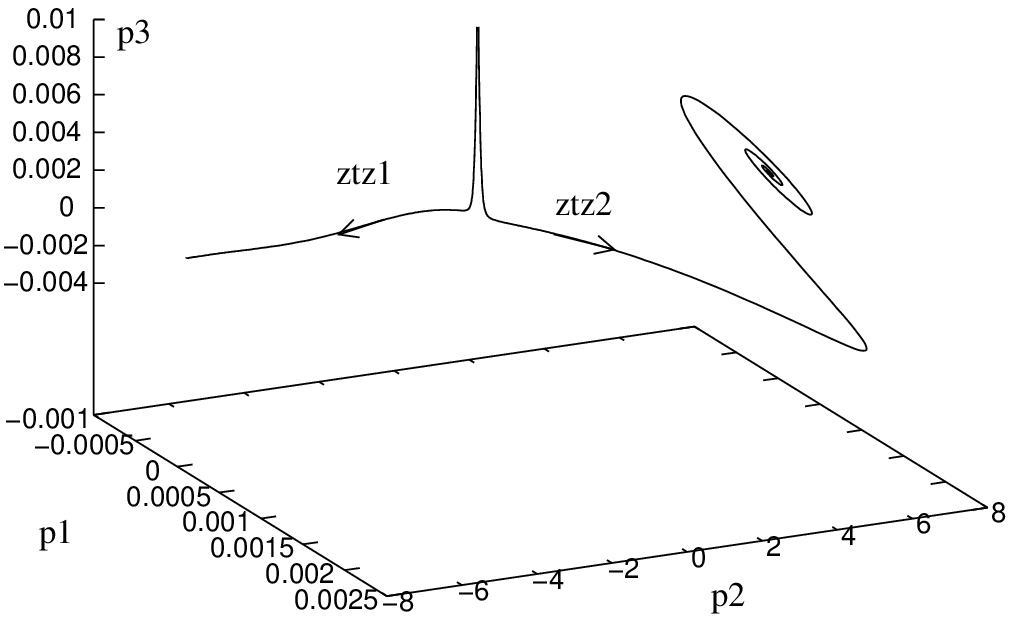}}
\subfloat[ ]{\includegraphics[scale=.5]{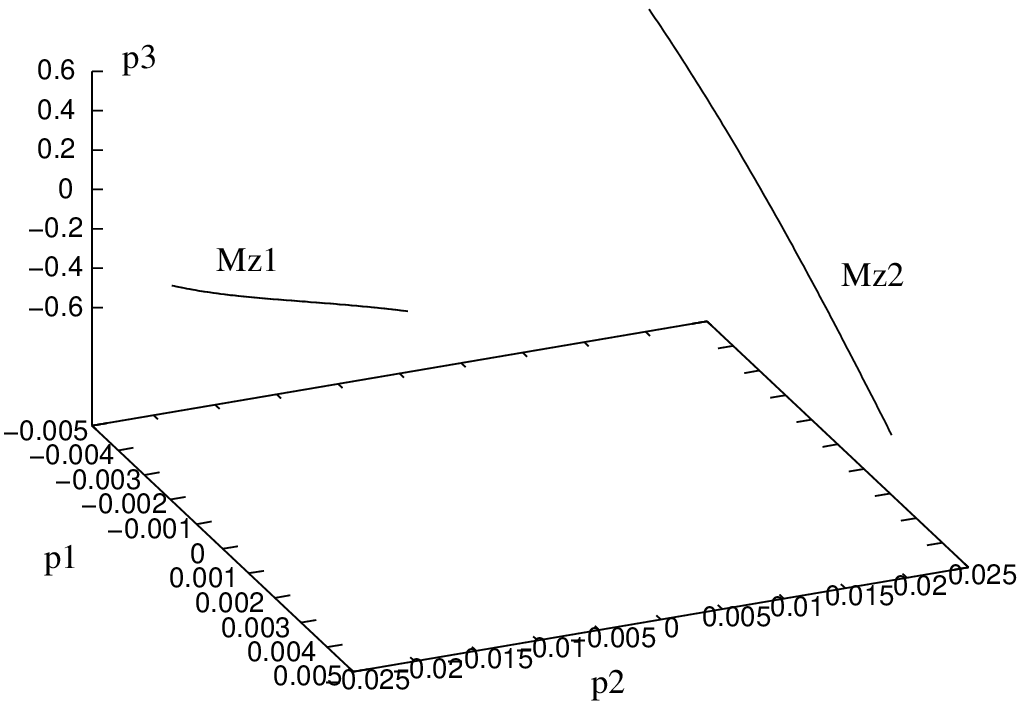}} \\
\subfloat[ ]{\includegraphics[scale=.5]{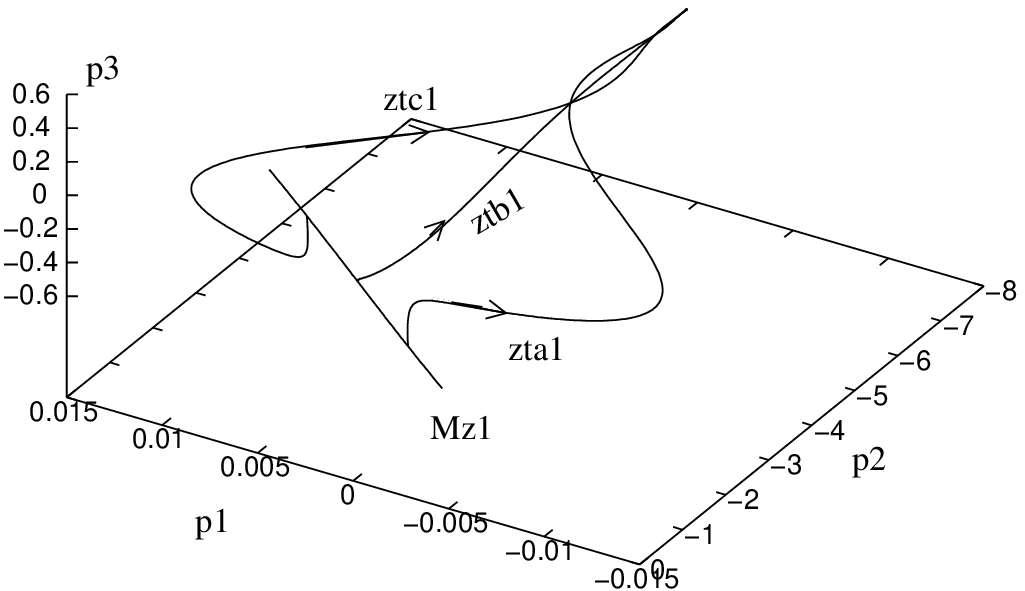}}
\subfloat[ ]{\includegraphics[scale=.5]{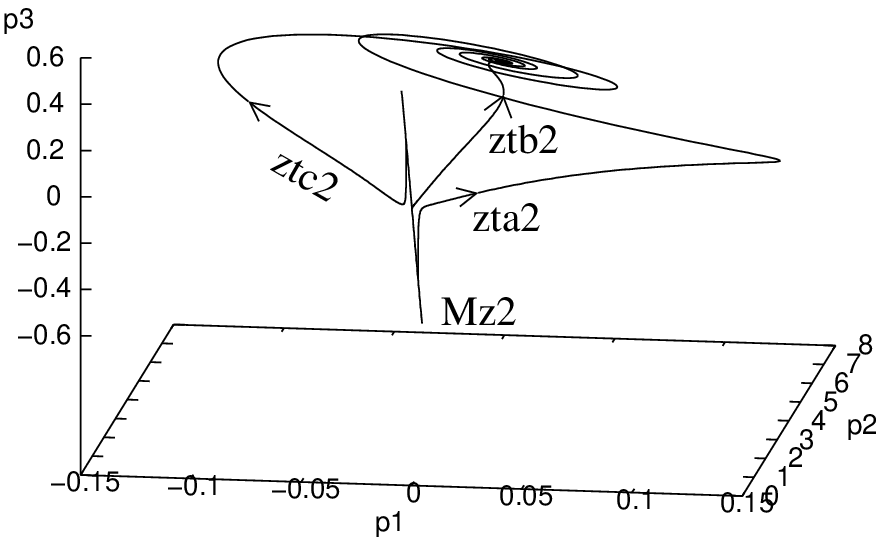}}
\caption{(A) The long time behavior of \eqref{AIFKSE equ} for two nearly different initial conditions, $z_1=(0, 0.02, 0.01)$ and $z_2=(0, -0.02, 0.01)$. (B) Leaves in the stable foliation through $z_1$ and $z_2$.  (C) $a_1$, $b_1$, $c_1\in\mathcal{M}_{z_1}$.  (D) $a_2$, $b_2$, $c_2\in\mathcal{M}_{z_2}$.  }
\label{fp_and_lc}
\end{center}
\end{figure}

\section{Comparison between the tracking initial condition and the projected one}
Given an initial condition $z_0$, it is natural to take $\tilde{z_0}:= (0,y_0)$, the linear projection onto the unstable eigenspace, as the approximate tracking initial condition since $Y$ is tangent to the inertial manifold.  However, the projected initial condition may have the wrong long time behavior as Figure \ref{lineartracking}(A) shows.  Moreover, even if it captures the correct long time behaviors, the tracking rate is not optimal compared to the tracking rate for the tracking initial condition.  By Proposition \ref{Exponential Tracking}, one has that
\begin{equation}
\label{tracking_estimate}
\|z(t,z_0)-z(t,z_0^+)\| \leq e^{(\alpha + \delta)t} \|x_0 - x_0^+\|,\;\forall\; t\geq0.
\end{equation}
Thus, the tracking rate for $z_0^+$ is $\alpha + \delta$.  To estimate for the tracking rate for $\tilde{z_0}$, we need so called {\it cone invariance property} in \cite{foliation}.   It is a stronger version of the squeezing property (where the gap condition does not hold), which was originally introduced for the Navier-Stokes equation in \cite{squeezing} and improved in \cite{FMT}.  

\begin{prop} (Cone Invariance Property)
\label{Cone Invariance Property}
If $z_1 \neq z_2$ and denote $u(t) = x(t,z_1) - x(t,z_2)$ and $v(t) = y(t,z_1) - y(t,z_2)$, then either
\begin{enumerate}[(i)]
\item{$\|v(t)\| < \|u(t)\|$, $\forall t \in \mathbb{R}$}
\item{$\|v(t)\| > \|u(t)\|$, $\forall t \in \mathbb{R}$}
\item{$\exists\; t_0 \in \mathbb{R}$ such that \begin{equation*} \begin{cases} \|v(t)\| < \|u(t)\|\;\;,\;\forall\; t<t_0\\\|v(t)\| = \|u(t)\|\;\;,\;\forall\; t=t_0\\\|v(t)\| > \|u(t)\|\;\;,\;\forall\; t>t_0\end{cases}\end{equation*}}
\end{enumerate}
\end{prop}

\begin{figure}[h!]
\psfrag{p1}{\tiny$p_1$}
\psfrag{p2}{\tiny$p_2$}
\psfrag{p3}{\tiny$p_3$}
\psfrag{ztz0}{\tiny$z(\cdot,z_0)$}
\psfrag{zty0}{\tiny$z(\cdot,\tilde{z_0})$}
\psfrag{ztzp}{\tiny$z(\cdot,z_0^+)$}
\psfrag{projected}{\tiny$\ln\|z(t,z_0)-z(t,\tilde{z_0})\| $}
\psfrag{tracking}{\tiny$\;\;\;\;\;\;\ln\|z(t,z_0)-z(t,z_0^+)\|$}

  \subfloat[]{\includegraphics[width=1.5in]{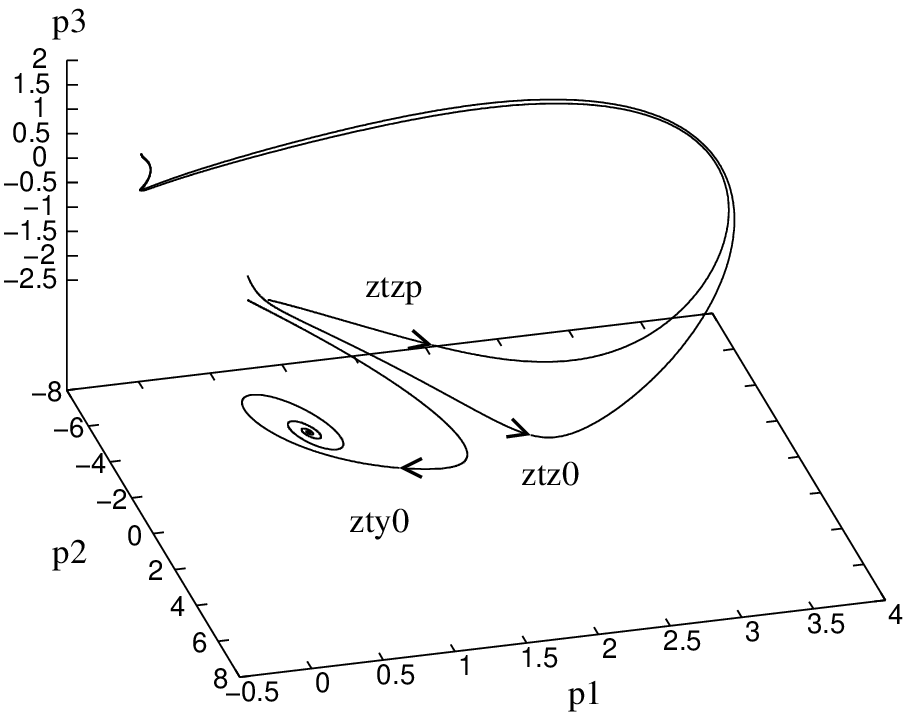}}
 \subfloat []{\includegraphics[width=1.5 in]{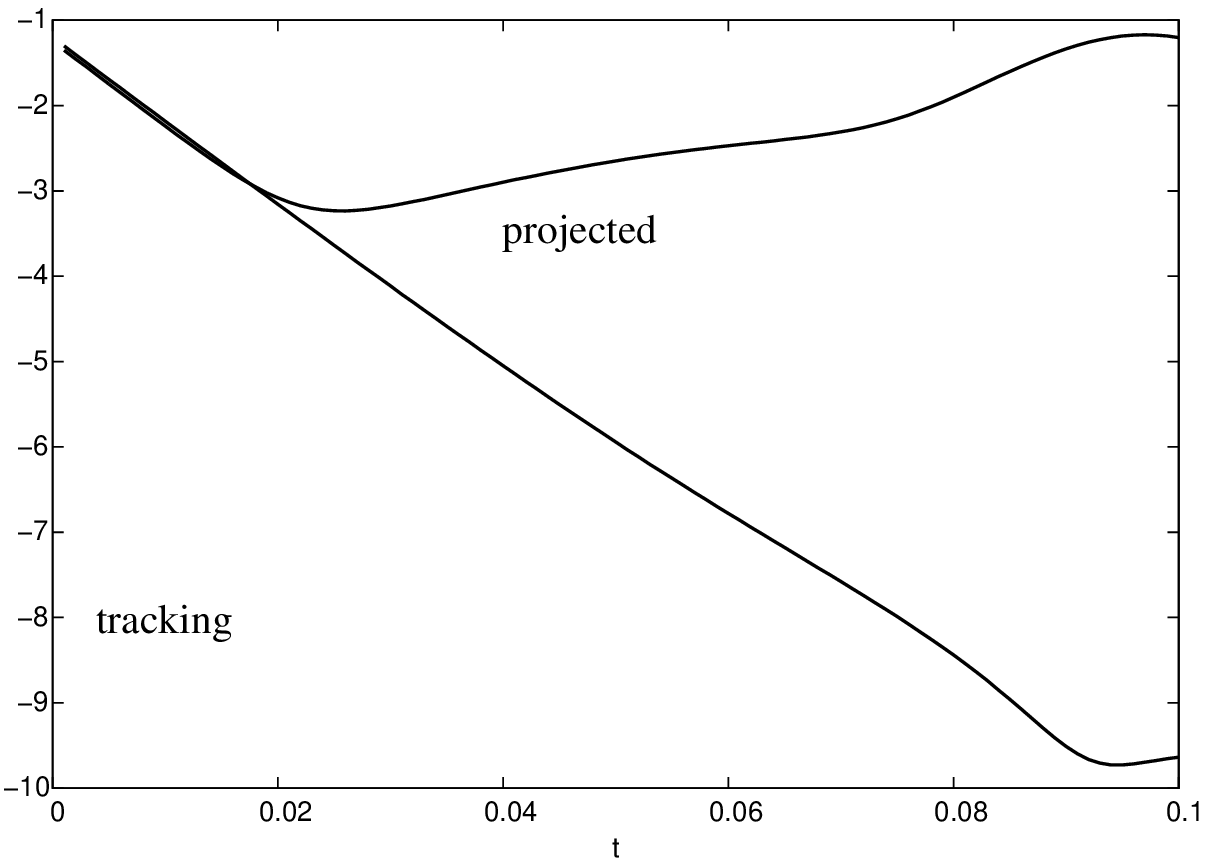} } 
 \subfloat []{\includegraphics[width=1.5 in]{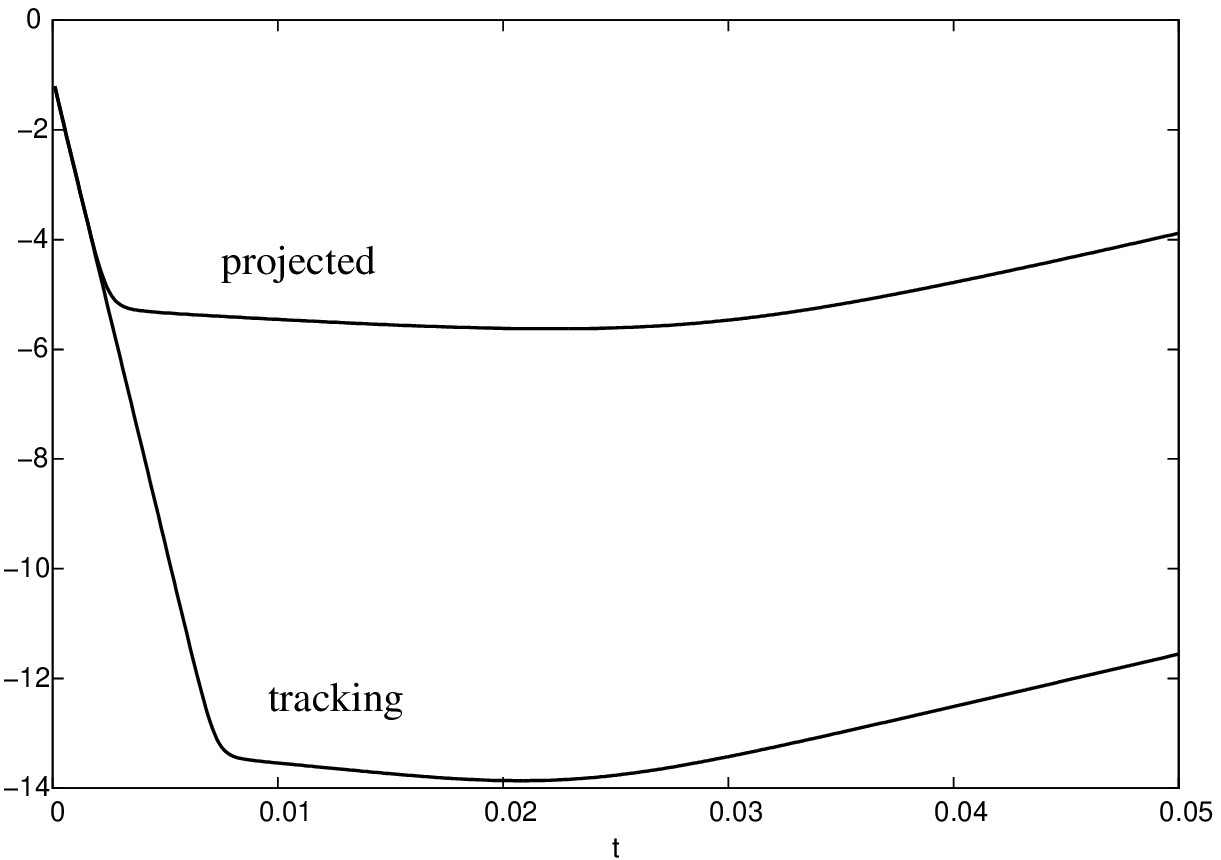} }

\caption{(A) For \eqref{AIFKSE equ}, $z_0=(0.12, 0.5, 0.5)$ and $\tilde{z_0} = (0.12, 0.5, 0.0)$ and the tracking initial condition $z_0^+=(0.2603, 0.5225, -0.00364)$, found as the fixed point of $\Sigma^{6}$.  (B) For \eqref{AIFKSE equ}.  (C) For $8$-mode KSE.}
\label{lineartracking}
 \end{figure}

%

\begin{lem}
\label{projected tracking}
For any $z_0\in Z$, either
 \begin{equation}
 \|z(t,z_0)-z(t,\tilde{z_0})\|\leq e^{(\alpha+\delta)t}\|x_0\|,\;\forall\; t\geq0,
 \end{equation}
or there exists a $t_0 \in (0,\infty)$ such that
\begin{equation}
\|z(t,z_0)-z(t,\tilde{z_0})\|\leq
  \begin{cases}
   e^{(\alpha+\delta)t}\|x_0\|,\;\forall\;t\in[0,t_0]. \\
   e^{((\beta-\delta)t+t_0(2\delta+\alpha-\beta))}\|x_0\|,\;\forall\;t\in(t_0,\infty).\\
  \end{cases}
 \end{equation}
\end{lem}
\begin{proof}
Denote $u(t)=x(t,z_0)-x(t,\tilde{z_0})$ and $v(t)=y(t,z_0)-y(t,\tilde{z_0})$.  It is clear that \\$0=\|v(0)|\leq\|u(0)\|=\|x_0\|$ and thus by Proposition \ref{Cone Invariance Property}, either $$\text{(i)}\;\;\|v(t)\|\le\|u(t)\|,\;\;\forall\; t\in\mathbb{R}$$ or (ii) there exists $t_0 \in \mathbb{R}$ such that \begin{equation*} \begin{cases} \|v(t)\| < \|u(t)\|\;\;,\;\forall\; t<t_0\\\|v(t)\| = \|u(t)\|\;\;,\;\forall\; t=t_0\\\|v(t)\| > \|u(t)\|\;\;,\;\forall\; t>t_0.\end{cases}\end{equation*}
By the variational of constants, one has
\begin{equation*}
x(t,z_0)=e^{tA}x_0 + \int_0^t e^{(t-s)A}F(z(s,z_0))\;ds,\;
\text{and}\;
x(t,\tilde{z_0})=\int_0^t e^{(t-s)A}F(z(s,\tilde{z_0}))\;ds.
\end{equation*}
If (i) holds, then
\begin{equation*}
\|z(t,z_0)-z(t,\tilde{z_0})\| = \|u(t)\| 
                      \leq e^{\alpha t} \|x_0\| + \delta\int_0^t e^{(t-s)A} \|u(s)\|\;ds.
\end{equation*}
By the Gronwall inequality, one has that \begin{equation*}\|u(t)\| \leq e^{(\alpha+\delta)t}\|x_0\|,\;\forall\;t\geq0. \end{equation*}
This proves the first part of the lemma.

Now if (ii) holds, for $t\leq t_0$, by the above estimate, we have
\begin{equation*}\|z(t,z_0)-z(t,\tilde{z_0})\| \leq e^{(\alpha+\delta)t}\|x_0\|.\end{equation*}
Using the Gronwall inequality and the fact that $\|u(t_0)\|=\|v(t_0)\|$ we have for $t>t_0$,
\begin{align*}
\|z(t,z_0)-z(t,\tilde{z_0})\| &= \|v(t)\|
                      \leq \|v(t_0)\|e^{(\beta-\delta)(t-t_0)} = \|u(t_0)\|e^{(\beta-\delta)(t-t_0)}\\
                      &\leq e^{(\alpha+\delta)t}\|x_0\|e^{(\beta-\delta)(t-t_0)}
                      =e^{(\beta-\delta)t+t_0(2\delta+\alpha-\beta)}\|x_0\|.\\
\end{align*}
\end{proof}

This lemma suggests that the tracking rate for $\tilde{z_0}$ is the same as the one for $z_0^+$ over a short period of time, namely $[0,t_0]$ for some $t_0$, but after that, the tracking rate becomes $(\beta-\delta)t+t_0(2\delta+\alpha-\beta)$.  It is easy to verify that $$(\alpha+\delta)t < (\beta-\delta)t+t_0(2\delta+\alpha-\beta),\;\forall t>t_0.$$


To show this result numerically, we consider both \eqref{AIFKSE equ} with $\gamma = 25$, $dim(X) = 1$ and $dim(Y)=2$ with the initial condition $z_0$ chosen at random within a ball of radius $0.4$ and the $8$-mode Galerkin approximation of the KSE with $\gamma=32$,  $dim(X)=4$, and $dim(Y)=4$ and the initial condition $z_0$ chosen at random from a ball within radius $0.5$.   Figure \ref{lineartracking}(B) and (C) compare the rates of attractions for the projected and tracking initial data.  Moreover, we compute their tracking rates via the simple linear regressions for these data sets in a short period of time ($0.1$ and $.008$ for \eqref{AIFKSE equ} and 8-mode Galerkin approximation of the KSE) and they are $-88.40$ and $12.59$ for \eqref{AIFKSE equ} and $-1640.45$ and $-400.04$ for 8-mode Galerkin approximation of the KSE.   The gap in eigenvalues is the interval $(-99,21)$ for \eqref{AIFKSE equ} and $(-1700,-512)$ for 8-mode Galerkin approximation of the KSE.  To ensure that this is not a special case, we repeat the above process with $10000$ randomly chosen initial conditions and plot the tracking rates as shown in Figure \ref{Tracking Rate}.  Observe that the data for tracking initial conditions are mainly close to the lower bound of the gap, while the data for projected initial conditions are mainly close to the upper bound of the gap.

\begin{figure}
\begin{center}
 \subfloat [\eqref{AIFKSE equ}]{\includegraphics[width=2. in]{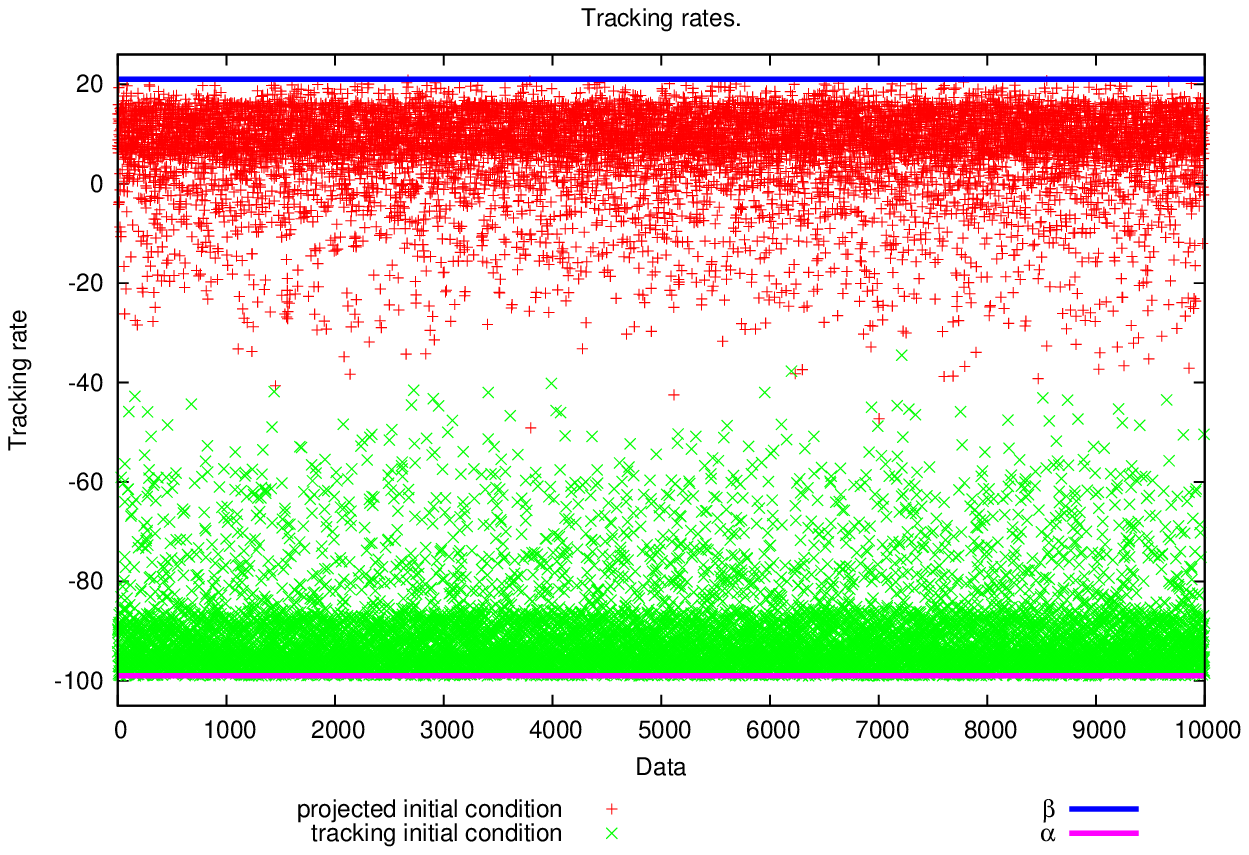} }
 \subfloat [$8$-mode KSE]{\includegraphics[width=2. in]{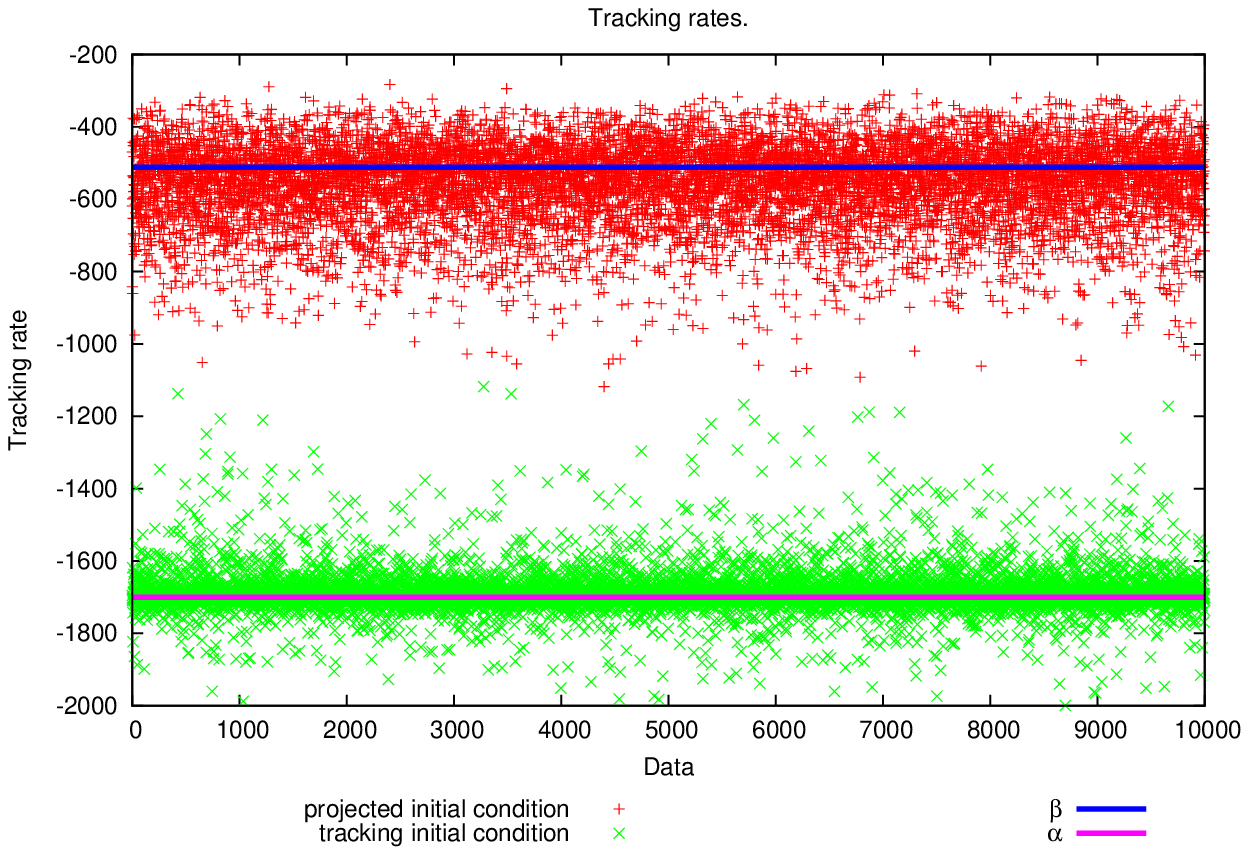} }
 \end{center}
\caption{Tracking rate for tracking initial condition ($\times$) is near the lower bound of the gap while the other ($+$) is near the upper bound of the gap.}
\label{Tracking Rate}
\end{figure}


\begin{thebibliography}{10}

\bibitem{foliation3}
P.~W. Bates and K.~Lu.
\newblock A {H}artman-{G}robman theorem for the {C}ahn-{H}illiard and
  {P}hase-{F}ield equations.
\newblock {\em J. of Dynamics and Differential Equations}, 6:101--145, 1994.

\bibitem{foliation_Bates_Lu_Zeng_1}
P.~W. Bates, K.~Lu, and C.~Zeng.
\newblock Existence and persistence of invariant manifolds for semiflows in
  {B}anach space.
\newblock {\em Mem. Amer. Math. Soc.}, 135(645):viii+129, 1998.

\bibitem{foliation_Bates_Lu_Zeng_2}
P.~W. Bates, K.~Lu, and C.~Zeng.
\newblock Invariant foliations near normally hyperbolic invariant manifolds for
  semiflows.
\newblock {\em Trans. Amer. Math. Soc.}, 352(10):4641--4676, 2000.

\bibitem{foliation}
N.~Castaneda and R.~Rosa.
\newblock Optimal estimates for the uncoupling of differential equation.
\newblock {\em Journal of Dynamics and Differential Equations}, 8:103--139,
  1996.

\bibitem{foliation_chow_lin_lu}
S.-N. Chow, X.-B. Lin, and K.~Lu.
\newblock Smooth invariant foliations in infinite-dimensional spaces.
\newblock {\em J. Differential Equations}, 94(2):266--291, 1991.

\bibitem{FOLI8}
Y.-M. Chung.
\newblock {FOLI8PAK} software package.
\newblock \url{http://php.indiana.edu/~msjolly/FOLI8PAK.html}

\bibitem{Roberts2}
S.~M. Cox and A.~J. Roberts.
\newblock Initial conditions for models of dynamical systems.
\newblock {\em Phys. D}, 85:126--141, 1995.

\bibitem{FMT}
C.~Foias, O.~Manley, and R.~Temam.
\newblock Modelling of the interaction of small and large eddies in
  two-dimensional turbulent flows.
\newblock {\em RAIRO Mod\'el. Math. Anal. Num\'er.}, 22(1):93--118, 1988.

\bibitem{inertialKSE}
C.~Foias, B.~Nicolaenko, G.~R. Sell, and R.~Temam.
\newblock Inertial manifolds for the {K}uramoto-{S}ivashinsky equation and an
  estimate of their lowest dimension.
\newblock {\em J. Math. Pures Appl. (9)}, 67(3):197--226, 1988.

\bibitem{inertial1}
C.~Foias, G.~R. Sell, and R.~Temam.
\newblock Inertial manifolds for nonlinear evolutionary equations.
\newblock {\em Journal of Differential Equations}, 73(2):309 -- 353, 1988.

\bibitem{inertial2}
C.~Foias, G.~R. Sell, and E.~S. Titi.
\newblock Exponential tracking and approximation of inertial manifolds for
  dissipative nonlinear equations.
\newblock {\em Journal of Dynamics and Differential Equations}, 1:199--244,
  1989.

\bibitem{squeezing}
C.~Foias and R.~Temam.
\newblock Some analytic and geometric properties of the solutions of the
  {N}avier-{S}tokes equations.
\newblock {\em J. Math. Pure Appl.}, 58:339--368, 1979.

\bibitem{AIFKSE}
M.~S. Jolly, I.~G. Kevrekidis, and E.~S. Titi.
\newblock Approximate inertial manifolds for the {K}uramoto-{S}ivashinsky
  equation: analysis and computations.
\newblock {\em Phys. D}, 44:38--60, August 1990.

\bibitem{center_manifold}
M.~S. Jolly and R.~Rosa.
\newblock Computation of non-smooth local centre manifolds.
\newblock {\em IMA Journal of Numerical Analysis}, 25:698--725, 2005.

\bibitem{JRT}
M.~S. Jolly, R.~Rosa, and R.~Temam.
\newblock Accurate computations on inertial manifolds.
\newblock {\em SIAM J. Sci. Comput.}, 22:2216--2238, June 2000.

\bibitem{dim_KSE_inertial}
M.~S. Jolly, R.~Rosa, and R.~Temam.
\newblock Evaluating the dimension of an inertial manifold for the
  {K}uramoto-{S}ivashinsky equation.
\newblock {\em Adv. Differential Equations}, 5:31--66, 2000.

\bibitem{foliation1}
U.~Kirchgraber and K.~Palmer.
\newblock Geometry in the neighborhood of invariant manifolds of maps and flows
  and linearization.
\newblock {\em Res. Notes Math.}, 233, 1990.

\bibitem{survey_dodel}
B.~Krauskopf, H.~M. Osinga, E.~J. Doedel, M.~E. Henderson, J.~Guckenheimer,
  A.~Vladimirsky, M.~Dellnitz, and O.~Junge.
\newblock {A Survey of Methods for Computing (un)Stable Manifolds of Vector
  Fields}.
\newblock {\em International Journal of Bifurcation and Chaos}, 15:763--791,
  2005.

\bibitem{foliation2}
K.~Lu.
\newblock A {H}artman-{G}robman theorem for reaction-diffusion equations.
\newblock {\em J. Diff. Eq.}, 93:364--394, 1991.

\bibitem{Ait}
Y.~Nievergelt.
\newblock Aitken's and {S}teffensen's accelerations in several variables.
\newblock {\em Numerische Mathematik}, 59:295--310, 1991.
\newblock 10.1007/BF01385782.

\bibitem{ComputingInertial4}
C.~Potzsche and M.~Rasmussen.
\newblock Computation of nonautonomous invariant and inertial manifolds.
\newblock {\em Numer. Math.}, 112(3):449--483, April 2009.

\bibitem{Roberts1}
A.~J. Roberts.
\newblock Appropriate initial conditions for asymptotic descriptions of the
  long term evolution of dynamical systems.
\newblock {\em J. Austral. Math. Soc. B.}, 31:48--75, 1989.

\bibitem{Roberts3}
A.~J. Roberts.
\newblock Computer algebra derives correct initial conditions for
  low-dimensional dynamical models.
\newblock {\em Computer Physics Communication}, 126:187--206, 2000.

\bibitem{ComputingInertial2}
J.~C. Robinson.
\newblock Computing inertial manifolds.
\newblock {\em Discrete and continuous dynamical systems}, 8:815--833, 2002.

\bibitem{rosa}
R.~Rosa.
\newblock Approximate inertial manifold of exponential order.
\newblock {\em Discrete and Continuous Dynamical System}, 1:421--448, 1995.

\bibitem{dim_KSE_inertial_Wang}
R.~Temam and X.~Wang.
\newblock Estimates on the lowest dimension of inertial manifolds for the
  {K}uramoto-{S}ivashinsky equation.
\newblock {\em Differential Integral Equations}, 7:1095--1108, 1994.
\end{thebibliography}
\end{document}